\documentclass[12pt]{article}

\usepackage{tikz}
\usetikzlibrary{intersections,calc,arrows.meta}
\usetikzlibrary{arrows.meta,decorations.pathmorphing,backgrounds,positioning,fit,petri,through}
\usetikzlibrary{datavisualization}
\usetikzlibrary{decorations.pathreplacing}
\usetikzlibrary{angles,quotes}

\usepackage{amssymb} 
\usepackage{amsthm}
\usepackage{amsmath}
\usepackage{latexsym}
\usepackage{mathrsfs}
\usepackage{enumerate}
\usepackage{graphicx}
\usepackage{fancybox}
\usepackage{color}
\usepackage{multicol, multienum}
\usepackage{bm}

\usepackage{mathtools}
\mathtoolsset{showonlyrefs=true}

\allowdisplaybreaks

\newtheorem{theorem}{Theorem}[section]
\newtheorem{lemma}[theorem]{Lemma}

\newtheorem{proposition}[theorem]{Proposition}

\theoremstyle{definition}
\newtheorem{definition}[theorem]{Definition}

\definecolor{mgre}{cmyk}{0.92,0.00,0.59,0.25}

\definecolor{PineGreen}{cmyk}{0.92,0.00,0.59,0.25}
\definecolor{ForestGreen}{cmyk}{0.91,0.00,0.88,0.12}
\definecolor{RawSienna}{cmyk}{0.00,0.72,1.00,0.45}
\definecolor{Mulbery}{cmyk}{0.34,0.90,0.00,0.02}
\definecolor{Sepia}{cmyk}{0.00,0.83,1.00,0.70}
\definecolor{Mahogany}{cmyk}{0.00,0.85,0.87,0.35}

\def\bE{{\mathbb{E}}}

\def\bP{{\mathbb{P}}}

\def\bR{{\mathbb{R}}}

\def\bfG{{\mathbf{G}}}

\def\bfM{{\mathbf{M}}}

\def\cB{{\mathcal{B}}}

\def\cE{{\mathcal{E}}}
\def\cF{{\mathcal{F}}}

\def\cH{{\mathcal{H}}}

\def\cK{{\mathcal{K}}}

\def\cM{{\mathcal{M}}}

\def\cS{{\mathcal{S}}}

\def\1{{\mathbf{1}}}

\newcommand{\loc}{\text{loc}}

\newtheorem{thm}{Theorem}[section]

\newtheorem{prop}[thm]{Proposition}
\newtheorem{lem}[thm]{Lemma}
\newtheorem{cor}[thm]{Corollary}

\theoremstyle{definition}
\newtheorem{defi}[thm]{Definition}
\newtheorem{prob}{}[section]
\newtheorem{rem}[thm]{Remark}
\newtheorem{ex}[thm]{Example}
\newtheorem{pro}{Problem}[section]
\newtheorem{ass}[thm]{Assumption}

\newcommand{\supp}{\mathrm{supp}}

\newcommand{\bd}{\begin{defi}}
\newcommand{\ed}{\end{defi}}

\newcommand{\bpro}{\begin{pro}}
\newcommand{\epro}{\end{pro}}

\newcommand{\bec}{\begin{cases}}
\newcommand{\eec}{\end{cases}}

\newcommand{\bpr}{\begin{prob}}
\newcommand{\epr}{\end{prob}}

\newcommand{\bt}{\begin{thm}}
\newcommand{\et}{\end{thm}}

\newcommand{\ba}{\begin{ass}}
\newcommand{\ea}{\end{ass}}

\newcommand{\br}{\begin{rem}}
\newcommand{\er}{\end{rem}}

\newcommand{\bpm}{\begin{pmatrix}}
\newcommand{\epm}{\end{pmatrix}}

\newcommand{\be}{\begin{ex}}
\newcommand{\ee}{\end{ex}}

\newcommand{\bp}{\begin{prop}}
\newcommand{\ep}{\end{prop}}

\newcommand{\bl}{\begin{lem}}
\newcommand{\el}{\end{lem}}

\newcommand{\bc}{\begin{cor}}
\newcommand{\ec}{\end{cor}}

\newcommand{\bq}{\begin{que}}
\newcommand{\eq}{\end{que}}

\newcommand{\beqn}{\begin{eqnarray*}}
\newcommand{\eeqn}{\end{eqnarray*}}

\newcommand{\beqnn}{\begin{eqnarray}}
\newcommand{\eeqnn}{\end{eqnarray}}

\newcommand{\bequ}{\begin{equation}}
\newcommand{\eequ}{\end{equation}}

\newcommand{\benu}{\begin{enumerate}}
\newcommand{\eenu}{\end{enumerate}}

\newcommand{\barr}{\begin{array}{rcl}}
\newcommand{\ear}{\end{array}}

\newcommand{\la}{\label}

\newcommand{\ds}{\displaystyle}

\title{Critical spectral behavior and large deviations for geometric $\alpha$-stable processes}
\author{
  \textsc{Kaneharu Tsuchida}\thanks{Department of Mathematics, National Defense Academy, Yokosuka, Kanagawa, Japan. \\
  tsuchida@nda.ac.jp} 
}
\date{\today}

\begin{document}
\numberwithin{equation}{section}
\maketitle

\begin{abstract}
  In this paper, we study the Schr\"odinger-type operator
  associated with geometric stable processes on $\mathbb{R}^{d}$,
  especially the differentiability of spectral function. Let $\mathcal{H}$ be
  the generator of the geometric stable process and $\mu$
  a smooth measure on $\mathbb{R}^{d}$.
  Then the spectral function $C(\theta)$ is defined as
  $C(\theta) = -\inf \sigma(-\mathcal{H} - \theta \mu)$,
  where $\sigma(\mathcal{A})$ denotes the spectrum of $\mathcal{A}$
  and $\theta$ is a real parameter. 
  Since the geometric stable process  exhibits severe local singularities
  in its L\'evy measure, 
  its transition semigroup lacks ultracontractivity,  
  which invalidates classical methods for proving the differentiability. 
  To overcome this obstacle, we use the compact embedding
  of the extended Dirichlet space into $L^2(\mu)$.
  As a primary application of this differentiability,
  we establish a large deviation principle
  for a positive continuous additive functional associated with
  the smooth measure $\mu$. 
  \end{abstract}
  \bigskip

   \medskip
  
  \noindent
  {\bf AMS 2020 \textit{Mathematics Subject Classification}}.
  Primary 60J45; Secondary 60J76, 31C25.

  \section{Introduction}
  Let $\mathcal{H}$ be a symmetric Markov generator and $\mu$ a smooth measure. 
  The spectral analysis of the Schr\"odinger type operator 
  $\mathcal{H}^{\theta\mu} = -\mathcal{H} - \theta\mu$ ($\theta \in \mathbb{R}$)
  is an important theme 
  in both mathematical physics and stochastic analysis. 
  It is particularly important for characterizing the transition 
  from subcriticality to criticality. 
  A fundamental object in this study is the spectral function
  $\theta \mapsto C(\theta)$, 
  defined by $C(\theta) = -\inf \sigma(\mathcal{H}^{\theta\mu})$, 
  where $\sigma(\cdot)$ denotes the spectrum in $L^2$. 
  Equivalently, it is the $L^2$-exponential growth rate 
  of the associated Feynman-Kac functional. 
  This spectral function characterizes the phase transitions of the system, 
  which depend on the coupling variable $\theta \in \mathbb{R}$. 
  There exists a critical coupling constant $\theta_c \ge 0$ such that for $\theta \le \theta_c$ 
  (the subcritical or critical regimes), the associated quadratic form 
  remains nonnegative, yielding $C(\theta) = 0$. 
  Conversely, for $\theta > \theta_c$ (the supercritical regime), 
  the form takes negative values. 
  In this case, the infimum of the spectrum becomes negative, 
  yielding $C(\theta) > 0$.

  The analytic properties of the spectral function $C(\theta)$
  depend heavily on these regimes. In the supercritical case ($\theta > \theta_{c}$),
  the principal eigenvalue is isolated from the essential spectrum $[0,\infty)$.
  Consequently, the standard analytic perturbation theory readily guarantees
  the spectral function is real analytic.
  However, as $\theta \downarrow \theta_{c}$, the principal eigenvalue is
  absorbed in the continuous spectrum, causing the complete breakdown
  of classical perturbation methods.
  Because the eigenvalue is no longer isolated,
  determining the differentiability of the spectral function 
  at the critical point $\theta_{c}$ requires
  deep potential and probability theoretic analysis. 

  The properties of $C(\theta)$ have been extensively
  investigated, tracing back to the foundational works on Schr\"odinger operators
  and Feynman-Kac semigroups by Kato, Simon and others.
  In the context of symmetric Markov processes, to properly situate the theoretical
  standing of our current results,
  it is instructive to detail the historical progression of the differentiability
  problem for the spectral function $C(\theta)$.
  
  To the best of our knowledge, the study on the differentiability of
  spectral function for Schr\"odinger operators with measure potentials was
  initiated in \cite{Takeda2003Large}.
  In this fundamental work, his framework was specifically focused on
  one- and two-dimensional Brownian motions perturbed by a positive measure
  potential belonging to the Green-tight Kato class. As is well known,
  Brownian motions in these low dimensions are recurrent, and
  more strongly, recurrent in the sense of Harris.
  A pivotal tool in his analysis was an inequality established by Oshima
  (ref. \cite{Oshima1982Potential923}), which is valid for general Harris
  recurrent symmetric Markov processes. We would like to note that
  in his argument, the null recurrence of the underlying processes
  was used implicitly to establish that the derivative strictly vanishes
  at the threshold, i.e., $\frac{dC}{d\theta} |_{(\theta = \theta_{c})} = 0$.
  Building upon this fundamental work,
  the analysis was extended to three- and four-dimensional Brownian motions
  in \cite{TakedaTsuchida2004}.
  Since Brownian motions in dimensions $d \ge 3$ are transient,
  the direct application of Oshima's inequality was no longer viable.
  To overcome this structural obstacle, the existence of a
  strictly positive, continuous, bounded ground state $h$
  associated with the critical Schr\"odinger operator 
  at $\theta = \theta_{c}$ was exploited. 
  It was demonstrated that by constructing the $h$-transformed process,
  the process successfully regains Harris recurrence.
  Crucially, it was in this framework that the underlying mechanism was
  brought to light: it became explicitly clear that the
  null recurrence of this $h$-transformed process is the indispensable key
  to proving that the derivative vanishes. Furthermore, it was 
  established that this necessary and sufficient condition for null criticality
  holds if and only if the spatial dimension is restricted to $d = 3$ or $4$.

  Subsequently, the theory was propelled beyond continuous trajectories
  into the realm of pure jump processes, specifically focusing on
  symmetric $\alpha$-stable processes in \cite{TakedaTsuchida2007}. 
  While this process reduces to a standard Brownian motion when
  $\alpha = 2$, for the case $0 < \alpha < 2$, it is pure jump-type
  Markov process.
  The transition to pure jump-type processes introduced a profound
  analytical obstacle: the construction and characterization of the ground state
  $h$ at the critical point $\theta = \theta_{c}$.
  Unlike the continuous case, where classical partial differential equation
  techniques can often guarantee the existence, strict positivity, and continuity
  of $h$, the jump case renders these standard local methods inapplicable.
  To overcome this non-trivial challenge,
  the interplay between the theory of Dirichlet forms
  and the potential theory of symmetric Markov processes was
  deeply exploited in \cite{TakedaTsuchida2007}. 
  By establishing some integral estimates specific to the non-local jump kernels,
  it was verified that this constructed function $h$ is indeed strictly positive,
  bounded and continuous.
  By analyzing the Green function, it was established that the null criticality of 
  the $h$-transformed process is equivalent to the condition $\alpha < d \le 2\alpha$. 
  For further extensions to broader classes of symmetric Markov processes, 
  we refer to \cite{Tsuchida2008Potential}. 
  These established frameworks have successfully resolved
  the critical behavior for a wide class of symmetric Markov processes.
  
  However, these classical methods encounter a severe obstacle
  when applied to the geometric $\alpha$-stable process (GSP). 
  The pioneering works mentioned above implicitly rely on
  ultracontractivity (a regularizing property) 
  in which the transition semigroup maps $L^2$ functions into $L^\infty$ spaces
  for any $t > 0$, as seen in Brownian motions and symmetric stable processes. 
  By contrast, the GSP, characterized by the logarithmic growth
  of its characteristic exponent $\Psi(\xi) = \log(1 + |\xi|^\alpha)$,
  lacks ultracontractivity. 
  This extreme path pathology is fundamentally rooted 
  in its L\'evy measure,
  which exhibits a highly irregular jumping mechanism compared to
  standard stable processes. 
  Specifically, its microscopic jumps near the origin are 
  too weak to instantly smooth out the spatial distribution, 
  while the process occasionally makes macroscopic jumps.
  Consequently, the particle intensely oscillates
  in a small area for long periods before making sudden large jumps.
  Interestingly, despite these severe local singularities,
  the GSP exhibits a macroscopic similarity to the standard stable process. 
  Because its characteristic exponent behaves
  as $\log(1 + |\xi|^\alpha) \sim |\xi|^\alpha$ as $|\xi| \to 0$,
  the Green function of the GSP shares the same asymptotic decay
  at spatial infinity. 
  This equivalence naturally leads to the conjecture
  that the threshold for null criticality should coincide
  with the condition $\alpha < d \le 2\alpha$ established
  for standard stable processes. 
  However, proving this expected result is highly nontrivial. 
  Classical methods for establishing null criticality
  and proving the differentiability of the spectral function $C(\theta)$
  heavily rely on constructing a continuous critical ground state $h$
  and using Harnack-type inequalities. 
  Deprived of ultracontractivity,
  these traditional approaches based on uniform boundedness
  and continuous approximations become completely inapplicable.

  To overcome this fundamental difficulty,
  our strategy relies on an array of analytic and probabilistic tools
  for the GSP that have been established over the past two decades. 
  First, we note that the work of \v{S}ik\'ic, Song and Vondra\v{c}ek
  \cite{SikicSongVondracek2006}. These results motivated our study. 
  They established the potential theory of GSPs, especially,
  Green function estimates, jump kernel estimates and Harnack inequalities
  using the theory of the subordinate Brownian motion and
  various Tauberian theorems.
  Next, we note the work of Takeda \cite{Takeda2019}.
  He introduced the notion of \textit{Class (T)}, and proved that
  if a symmetric Markov process satisfies the Class (T),
  the embedding of the Dirichlet space into the underlying $L^{2}$-space
  is compact, and its ground state has a strictly positive and bounded
  continuous modification. These results are used in various parts of this paper.
  Recently, we proved the absolute continuity of the GSP using a quite probabilistic
  method in \cite{tsuchida2026GSPSF}. Since the GSP is a L\'evy process,
  the absolute continuity
  implies the strong Feller property.
  Combining these results, we overcome the lack of ultracontractivity and
  can prove the differentiability of the spectral function.

  It is known that when the spectral bound of
  the Schr\"odinger operator is $L^{p}$ independent,
  this spectral function coincides with the logarithmic moment generating function
  of the positive continuous additive functional (PCAF)
  generated by the measure potential.
  Because the regularity of the transition density is not available in our setting,
  the $L^{p}$-independence requires a delicate treatment. 
  Here we use the result of Chen \cite{chen2012lp}. 
  By this $L^{p}$-independence, we show that $C(\theta)$ becomes
  the logarithmic moment generating function, hence we obtain
  the large deviation principle for the additive functional using
  the differentiability of $C(\theta)$ and the G\"artner-Ellis theorem. 

  The rest of this paper is organized as follows. 
  In Section 2, we prepare some basic definitions and preliminary facts 
  concerning the GSP and Kato class measures. 
  In Section 3, we characterize the critical coupling constant $\theta_{c}$ 
  and establish the existence of the corresponding ground state. 
  In Section 4, we prove the differentiability of the spectral function, 
  which plays a central role in our analysis. 
  Finally, in Section 5, we apply the differentiability result 
  to establish the large deviation principle for 
  the positive continuous additive functionals.

  Throughout this paper, we use the following notation. 
The symbol $\mathbb{R}^d$ denotes the $d$-dimensional Euclidean space, 
and $B(R)$ denotes the open ball in $\mathbb{R}^d$ of radius $R$ centered at the origin. 
For $p \in [1, \infty]$, the space $L^p(\mathbb{R}^d, dx)$ stands for the Lebesgue space 
equipped with the $L^{p}$-norm $\|\cdot\|_{p}$. 
We denote by $\mathcal{B}(\mathbb{R}^d)$ and $C(\mathbb{R}^d)$ the spaces 
of Borel functions and continuous functions on $\mathbb{R}^d$, respectively. 
The notations $C_0(\mathbb{R}^d)$ and $C_0^\infty(\mathbb{R}^d)$ denote the spaces 
of continuous functions and infinitely differentiable functions 
with compact support in $\mathbb{R}^d$, respectively. 
For a given function space, we use the subscript ``b'' to denote the subspace 
of bounded functions and the superscript ``+'' to denote the subspace 
of non-negative functions (e.g., $\mathcal{B}_b^+(\mathbb{R}^d)$ stands for the space 
of non-negative bounded Borel functions on $\mathbb{R}^d$). 
For a Markov process $(X_t)_{t \ge 0}$, we write $\mathbb{P}_x$ for the probability law 
of the process starting at $x \in \mathbb{R}^d$, and $\mathbb{E}_x$ for the corresponding expectation. 
For positive functions $a(r)$ and $b(r)$, the relation ``$a(r) \sim b(r)$ as $r \to a$'' 
means that $\lim_{r \to a} a(r)/b(r) = 1$. 
Furthermore, ``$a(r) \asymp b(r)$ as $r \to a$'' means that 
$c_{1}b(r) \le a(r) \le c_{2} b(r)$ near $a$ for some constants $c_{2} > c_{1} > 0$. 
We use $c, C, C_{1}, C_{2}, \dots$ to denote positive constants 
which may be different at different occurrences.
  
  \section{Preliminaries}
  In this section, we prepare some basic definitions and properties of the
  geometric stable process.
  
  Let $\bfG^{\alpha}
  = (\Omega, \cM, \cM_{t}, \theta_{t}, \bP_{x}, X_{t})$ be
  a geometric $\alpha$-stable process on $\bR^{d}$.
  Here, $(\Omega,\cM, \bP_{x})$ is an underlying probability space
  for the process starting at $x$, 
  $\{\cM_{t}\}$ is the minimal augmented filtration, and
  $\theta_{t} : \Omega \to \Omega$ is the shift operator satisfying
  $X_{s} (\theta_{t} \omega)  = X_{t+s}(\omega)$. 
  The process $X_{t}$ is characterized by its characteristic function:
  \begin{align}
  \label{eq:1}
    \bE_{0}\left[ e^{i \xi \cdot X_{t}} \right] = e^{-t \psi(\xi)}, \quad \xi \in \bR^{d}, 
  \end{align}
  where the characteristic exponent (symbol) $\psi$ is given by
  \begin{align}
  \label{eq:2}
  \psi^{\alpha}(\xi) = \log{(1 + |\xi|^{\alpha})}, \quad \alpha \in (0,2].
  \end{align}
  The infinitesimal generator $\cH^{\alpha}$ of the process $\bfG^{\alpha}$ is
  a pseudo-differential operator with the symbol $-\psi^{\alpha}(\xi)$,
  which is formally written as
  \begin{align}
  \label{eq:3}
  \cH^{\alpha} = - \log{(1 + (-\Delta)^{\alpha/2})}.
  \end{align}
  According to the Chung-Fuchs-type criterion, $\bfG^{\alpha}$ is transient for
  $d > \alpha$, and recurrent for $d \le \alpha$.
  Let $P_t$ and $p_t(x, dy)$ be the transition semigroup
  and the transition probability of the GSP $\bfG^{\alpha}$, respectively. 
  Recently, it has been established in \cite{tsuchida2026GSPSF}
  that for all $t > 0$, $\bfG^{\alpha}$ admits
  a transition density $p_{t}(x, y) (= p_{t}^{0}(|x-y|))$ 
  with respect to the Lebesgue measure, where $p_{t}^{0}(z)$ denotes the
  transition density of $\bfG^{\alpha}$ starting at the origin. 
  Since $\bfG^{\alpha}$ is a L\'evy process, we see that for $f \in \cB_{b}(\bR^d)$
  \begin{align}
  \label{eq:19}
    P_{t}f(x) = \bE_{x}[f(X_{t})] = \int_{\bR^{d}}^{ } p_{t}(x,y) f(y) dy
    = \int_{\bR^{d}}^{ }  p_{t}^{0}(y-x)f(y) dy. 
  \end{align}
  For any $f \in \cB_{b}(\bR^{d})$, the action of the transition semigroup
  can be written as the convolution $P_{t} f = p_{t} \ast f$.
  Since $p_{t}(x,\cdot) \in L^{1}(\bR^{d})$,
  the smoothing effect of this convolution ensures
  the strong Feller property, that is, $P_{t} (\cB_{b}) \subset C_{b}$
  (see also \cite{Hawkes1979Potential}). 
  This strong Feller property allows us to refine various
  potential-theoretic properties from quasi-everywhere (q.e.)
  statements to pointwise statements,
  which will be crucially used in the subsequent sections.
  
  In the sequel, we assume $d > \alpha$, that is, $\bfG^{\alpha}$ is transient
  and we put the Green function as 
  \begin{align}
  \label{eq:4}
  G(x,y) := \int_{0}^{\infty} p_{t}(x,y) dt.
  \end{align}
  \v{S}iki\'c, Song and Vondra\v{c}ek developed the potential theory for
  geometric stable process and obtained asymptotic estimates
  for the Green function and the L\'evy density. We summarize their
  results in the following two theorems.

  \begin{theorem}[{\cite[Theorem 3.2, Remark 3.3]{SikicSongVondracek2006}}]
    \label{thm:Green}
    For any $\alpha \in (0,2]$,
    \begin{align}
      G(x,y) &\sim
               \begin{dcases}
                 \frac{\Gamma(d/2)}{2\alpha \pi^{d/2}|x-y|^{d} \log^{2}{\frac{1}{|x-y|}}},
                 & |x-y| \to 0, \\
                 \frac{1}{\pi^{d/2}2^{\alpha}}
                 \frac{\Gamma(\frac{d-\alpha}{2})}{\Gamma(\frac{\alpha}{2})}
                 |x-y|^{\alpha - d}, & |x-y| \to \infty. \label{eq:5} 
               \end{dcases}
      \end{align}
  \end{theorem}

  \begin{theorem}[{\cite[Theorem 3.4, 3.5]{SikicSongVondracek2006}}]
    \label{thm:Levy}
    For $\alpha \in (0,2)$,
    \begin{align}
    \label{eq:47}
      J(x,y) \sim
      \begin{dcases}
        \frac{\alpha \Gamma(d/2)}{2|x-y|^{d}}, & |x-y| \to 0, \\
       \frac{\alpha \Gamma(d/2)}{2^{\alpha + 1} \pi^{d/2}
      \Gamma(1 - \alpha/2)} \frac{1}{|x-y|^{d+\alpha}}, & |x-y| \to \infty.
      \end{dcases}
    \end{align}
    For $\alpha = 2$, 
    \begin{align}
    \label{eq:7}
      J(x,y) \sim
      \begin{dcases}
        \frac{\Gamma(d/2)}{|x-y|^{d}}, & |x-y| \to 0, \\
        2^{-d/2} \pi^{-\frac{d-1}{2}} \frac{e^{-|x-y|}}{|x-y|^{\frac{d+1}{2}}}, & |x-y| \to \infty
      \end{dcases}
    \end{align}
  \end{theorem}

  Define the following quadratic form:
  \begin{align}
    \label{eq:13}
    \begin{dcases}
     \cE^{\alpha}(u,v) = \int_{\bR^{d}}^{ } \widehat{u}(\xi) \bar{\widehat{v}}(\xi)
    \log{(1 + |\xi|^{\alpha})} d\xi, \\
    \cF^{\alpha} = \left\{ u \in L^{2}(\bR^{d}) : \int_{\bR^{d}}^{ } \left| \widehat{u}(x) \right|^{2}
    \log{(1 + |\xi|^{\alpha})} d\xi < \infty\right\}.
    \end{dcases}
  \end{align}
  Then it is easy to see that $(\cE^{\alpha}, \cF^{\alpha})$ is 
  the regular Dirichlet form associated with $\bfG^{\alpha}$,
  having $C_{0}^{\infty}(\mathbb{R}^{d})$ as its core 
  (see \cite[Example 1.4.1]{FOT}).
  Moreover, since $\bfG^{\alpha}$ is a conservative L\'evy process with
  purely jump paths, we find that
  \begin{align}
  \label{eq:14}
    \cE^{\alpha}(u,v) = \frac{1}{2}\iint_{\bR^{d} \times \bR^{d} \setminus d}
    (u(x) - u(y)) (v(x) - v(y))
    J(x,y) dxdy,
  \end{align}
  where $J(x,y)$ is the L\'evy density of $\bfG^{\alpha}$
  (see Theorem \ref{thm:Levy}) and $d = \{(x,x): x \in \bR^{d}\}$. 
  We denote by $\cF_{e}^{\alpha}$ the family of
  $m$-measurable function $u$ on $\bR^{d}$
  such that $|u| < \infty$ a.e. and there exists an $\cE^{\alpha}$-Cauchy
  sequence $\{u_{n}\}$
  of functions in $\cF^{\alpha}$ such that $\lim_{n \to \infty} u_{n} = u$ a.e.
  We refer to $(\cF_{e}^{\alpha}, \cE^{\alpha})$ as the extended Dirichlet space.
  If $\bfG^{\alpha}$ is transient, $(\cF_{e}^{\alpha}, \cE^{\alpha})$
  is a Hilbert space with inner product $\cE^{\alpha}$.

  We omit the standard definitions related to potential theory 
  and Markov processes associated with regular Dirichlet forms 
  (e.g., capacity, quasi-continuity, smooth measures, 
  and additive functionals), 
  and refer the reader to \cite{FOT} for comprehensive details
    
  Throughout this paper, we always identify each element of
  $\mathcal{F}_{e}$ with its quasi-continuous version.
  It is well-known that there exists one-to-one correspondence between
  the space of smooth measures $\cS$ and the one of positive continuous
  additive functionals (up to equivalence). This correspondence
  is referred to ``Revuz correspondence'',
  for any $f \in \cB^{+}$ and $\gamma$-excessive function $h$, 
  \begin{align}
  \label{eq:12}
    \lim_{t \to 0} \frac{1}{t} \bE_{hm}\left[ \int_0^{ t} f(X_{s}) dA_{s}^{\mu} \right]
    = \int_{\bR^{d}}^{ } f(x) h(x) \mu(dx).
  \end{align}
  In this paper, we only consider positive continuous additive functionals in
  the strict sense. 
  
  Next, we define the Kato class and its associated Green-tight measures.
  This Green-tight measure is originally introduced
  in Chen \cite{Chen:2002:Gaugeability}. 
  For a positive smooth measure $\mu$, its potential $G\mu$ is
  defined by
  \begin{align}
  \label{eq:9}
    G \mu(x) := \int_{\mathbb{R}^{d}}^{ } G(x,y) \mu(dy) = \bE_{x}[A_{\infty}^{\mu}],
  \end{align}
  where $A^{\mu}$ is the positive continuous additive functional with
  the Revuz measure $\mu$. 

  \begin{definition}
    \la{def:Kato}
    \begin{enumerate}[(1)]
    \item A smooth measure $\mu \in \cS$ is said to be in Kato class ($\cK$) if
      the associated PCAF $A_{t}^{\mu}$ satisfies
      \begin{align}
      \label{eq:8}
      \lim_{t \to 0} \sup_{x \in \bR^{d}} \bE_{x}[A_{t}^{\mu}] = 0. 
      \end{align}
    \item A smooth measure $\mu \in \cS$ is said to be 
    in the class $\cK_{\infty}$
      if for any $\varepsilon > 0$, there is a Borel set $K = K(\varepsilon)$ of
      finite $\mu$-measure and a constant $\delta = \delta(\varepsilon) > 0$
      such that for all measurable sets $B \subset K$ with $\mu(B) < \delta$,
      \begin{align}
      \label{eq:6}
        \sup_{x \in \mathbb{R}^{d}} G(1_{B \cup K^{c}} \mu)(x) < \varepsilon.
      \end{align}
     \item A smooth measure $\mu$ is said to be in the class $\cS_{\infty}$
       if for any $\varepsilon > 0$ there exists a Borel subset $K = K(\varepsilon)$
       of finite $\mu$-measure and a constant $\delta = \delta(\varepsilon) > 0$
       such that for all measurable sets $B \subset K$ with $\mu(B) < \delta$,
       \begin{align}
       \label{eq:10}
         \sup_{(x,z) \in \bR^{d} \times \bR^{d}} \int_{B \cup K^{c}}^{ }
         \frac{G(x,y) G(y,z)}{G(x,z)} \mu(dy) \le \varepsilon.
       \end{align}
    \end{enumerate}
  \end{definition}

  We summarize known results on relations among 
  $\mathcal{K}, \mathcal{K}_{\infty}, \mathcal{S}_{\infty}$ 
  and a property.

  \begin{lemma}
    \label{lem:relation-K}
    \begin{enumerate}[{\rm (i)}]
    \item {\rm (\cite[Proposition 2.3]{Chen:2002:Gaugeability})}
      $\mathcal{K}_{\infty} \subset \mathcal{K}$.
    \item {\rm (\cite[Corollary 3.1]{ChenSong2002General})}
      $\mathcal{S}_{\infty} \subset \mathcal{K}_{\infty}$. 
      \item {\rm (\cite[Proposition 2.2]{Chen:2002:Gaugeability})} 
      For $\mu \in \mathcal{K}_{\infty}$, $G\mu$ is bounded.
    \end{enumerate}
  \end{lemma}

  In fact, we can strengthen Lemma \ref{lem:relation-K} (ii) in the
  case of GSPs. The following lemma plays a key role. 
  
  \begin{lemma}[Global 3G inequality]
    \label{lem:3G}
    Suppose $G(x,y)$ is the global Green function of a geometric stable process on $\mathbb{R}^d$.
    Then there exists a constant $C > 0$ such that for all $x, y, z \in \mathbb{R}^d$,
    \begin{align}
    \label{eq:38}
      \frac{G(x,y)G(y,z)}{G(x,z)} \le C \left( G(x,y) + G(y,z) \right).
    \end{align}
\end{lemma}
\begin{proof}
  Since the process is spatially homogeneous and isotropic and
  subordinate Brownian motion,
  its Green function can be written as $G(x,y) = g(|x-y|)$,
  where $g: (0, \infty) \to (0, \infty)$ is a strictly decreasing function.

  We first show that there exists a constant $C_0 > 0$ such that
  \begin{align}
  \label{eq:double}
    g(r) \le C_0 g(2r) \quad \text{for all $r > 0$}.
  \end{align}
  By Theorem \ref{thm:Green}, we know that 
  \begin{align}
    \label{eq:58}
    g(r) &\asymp
           \begin{dcases}
             r^{-(d-\alpha)} \quad &\text{for } r \ge 1/2, \\
             r^{-d} \left( \log\frac{1}{r} \right)^{-2} \quad &\text{for } r < 1/2.
           \end{dcases}
  \end{align}
  For the long-range behavior ($r \ge 1/2$), we have:
  \begin{align}
  \label{eq:39}
  \frac{g(r)}{g(2r)} \asymp \frac{r^{-(d-\alpha)}}{(2r)^{-(d-\alpha)}} = 2^{d-\alpha} < \infty.
  \end{align}
  For the short-range behavior ($r < 1/2$), we evaluate the ratio as follows:
  \begin{align}
  \label{eq:55}
  \frac{g(r)}{g(2r)} \asymp \frac{r^{-d} \left( \log\frac{1}{r} \right)^{-2}}{(2r)^{-d} \left( \log\frac{1}{2r} \right)^{-2}} 
    = 2^d \left( \frac{\log\frac{1}{2r}}{\log\frac{1}{r}} \right)^2.
  \end{align}
  Noting that $\log(1/(2r)) = \log(1/r) - \log 2$,
  the term inside the parenthesis converges to $1$ as $r \to 0$.
  Thus, the ratio $\frac{g(r)}{g(2r)}$ remains bounded near the origin.
  In the intermediate region, the ratio is trivially bounded because $g$ is positive and continuous.
  Consequently, this shows that $g(r)$ satisfies the doubling condition
  globally with some constant $C_0 > 0$.
  
  By symmetry, we may assume without loss of generality that $|x-y| \ge |y-z|$.
  Using the triangle inequality, we obtain:
  \begin{align}
  \label{eq:56}
      |x-z| \le |x-y| + |y-z| \le 2|x-y|.
  \end{align}
  Since $g$ is a decreasing function and the doubling condition \eqref{eq:double},
  we see
  \begin{align}
  \label{eq:57}
    g(|x-z|) \ge g(2|x-y|) \ge \frac{1}{C_{0}} g(|x-y|).
  \end{align}
  Hence we obtain:
  \begin{align}
  \label{eq:70}
    g(|x-z|) \ge \frac{1}{C_0} g(|x-y|) \implies \frac{g(|x-y|)}{g(|x-z|)} \le C_0.
  \end{align}

  Using this, we can evaluate the conditional kernel:
  \begin{align}
  \label{eq:71}
    \frac{g(|x-y|)g(|y-z|)}{g(|x-z|)} = \left( \frac{g(|x-y|)}{g(|x-z|)} \right) g(|y-z|) \le C_0 \, g(|y-z|).
  \end{align}
  Since the Green function is strictly positive,
  adding $g(|x-y|)$ to the right-hand side strictly preserves the inequality.
  Therefore we show that 
  \begin{align}
  \label{eq:72}
    \frac{G(x,y)G(y,z)}{G(x,z)}
    &= \frac{g(|x-y|) g(|y-z|)}{g(|x-z|)} \\
    &\le C_{0} g(|y-z|) \le C_{0}(g(|x-y|) + g(|y-z|)) \\
    &\le C_{0}(G(x,y) + G(y,z)),
  \end{align}
  which completes the proof.
\end{proof}

\begin{proposition}
  \label{prop:KS}
  The class $\mathcal{S}_{\infty}$ coincides with $\mathcal{K}_{\infty}$ for the GSP. 
\end{proposition}
\begin{proof}
This is clear by Lemma \ref{lem:relation-K} and \ref{lem:3G}.
\end{proof}

The following theorem, a Poincar\'e-type inequality first established
by Stollmann and Voigt \cite{StollmannVoigt1996}, plays a crucial role in the subsequent discussion.

\begin{theorem}
  \label{thm:SV}
  Let $G_{\beta}$ be the $\beta$-potential operator, that is,
\begin{align}
\label{eq:11}
  G_{\beta} \mu(x)
  &:= \int_{\mathbb{R}^{d}}^{ } G_{\beta}(x,y) \mu(dy)
    := \int_{\mathbb{R}^{d}}^{ } \left( \int_0^{ \infty} e^{-\beta t} p_{t}(x,y) dt\right) \mu(dy).
\end{align}
  If $\mu \in \mathcal{K}$, then for $\beta \ge 0$, 
  \begin{align}
  \label{eq:SV}
    \int_{\mathbb{R}^{d}}^{ } u(x)^{2} \mu(dx) \le \left\| G_{\beta} \mu \right\|_{\infty}
    \mathcal{E}_{\beta}^{\alpha}(u,u),\quad u \in \mathcal{F}^{\alpha}\
    (u \in \mathcal{F}_{e}^{\alpha}\ \text{if $\beta = 0$}),
  \end{align}
  where $\mathcal{E}_{\beta}^{\alpha}(u,u)
  = \mathcal{E}^{\alpha}(u,u) + \beta \int_{\mathbb{R}^{d}}^{ } u(x)^{2} dx$.
\end{theorem}

It is known that $\mu \in \mathcal{K}$ if and only if
\begin{align}
\label{eq:73}
  \lim_{\beta \to \infty} \left\| G_{\beta} \mu \right\|_{\infty} = 0. 
\end{align}
Therefore, the inequality \eqref{eq:SV} implies that
for any $\varepsilon > 0$ there exists $M(\varepsilon) > 0$ such that
\begin{align}
  \label{eq:SV-1}
  \int_{\mathbb{R}^{d}}^{ } u(x)^{2} \mu(dx) \le
  \varepsilon \mathcal{E}^{\alpha}(u,u)
  + M(\varepsilon) \int_{\mathbb{R}^{d}}^{ } u(x)^{2} dx, \quad u \in \mathcal{F}^{\alpha}. 
\end{align}
It is known in \cite[Proposition 2.2]{Chen:2002:Gaugeability} that
$\mu \in \mathcal{K}_{\infty}$ is Green bounded:
\begin{align}
\label{eq:SV-2}
  \sup_{x \in \mathbb{R}^{d}} G\mu(x)
  = \sup_{x \in \mathbb{R}^{d}} \bE_{x}[A_{\infty}^{\mu}] < \infty. 
\end{align}
Hence the inequality \eqref{eq:SV} yields that for any $\mu \in \mathcal{K}_{\infty}$,
\begin{align}
\label{eq:69}
  \int_{\mathbb{R}^{d}}^{ } u(x)^{2} \mu(dx) \le \left\| G\mu \right\|_{\infty}
  \mathcal{E}^{\alpha}(u,u), \quad u \in \mathcal{F}_{e}^{\alpha}, 
\end{align}
in particular, $\mathcal{F}_{e}^{\alpha}$ is contained in $L^{2}(\mu)$.

In the sequel, we fix a measure $\mu \in \mathcal{K}_{\infty}$ with $\mu \not\equiv 0$.
  We define the spectral function. 
  \begin{definition}
  \label{def:sf}
  For $\theta \in \bR$, the spectral function $C(\theta)$
  associated with the measure $\mu$ is defined by
  \begin{equation}
  \label{eq:15}
    C(\theta) = - \inf \left\{ \cE^{\alpha}(u,u) - \theta \int_{\bR^{d}} u(x)^{2} \mu(dx)
    \;:\; u \in \cF^{\alpha}, \int_{\bR^{d}} u(x)^{2} dx = 1 \right\}.
  \end{equation}
\end{definition}

It is clear that the spectral function $C(\theta)$ has the following properties.
\begin{lemma}
  \label{lem:spectral_properties}
  The spectral function $C(\theta)$ is a non-negative, convex, and continuous function on $\bR$.
\end{lemma}
\begin{proof}
  By pulling the minus sign inside the infimum, we can rewrite the definition of $C(\theta)$ as a supremum:
  \begin{align}
  \label{eq:spectral_sup}
    C(\theta) = \sup \left\{ \theta \int_{\bR^{d}} u(x)^{2} \mu(dx) - \cE^{\alpha}(u,u) 
    : u \in \cF^{\alpha}, \int_{\bR^{d}} u(x)^{2} dx = 1 \right\}.
  \end{align}
  For each fixed function $u \in \cF^{\alpha}$ with $\|u\|_{2} = 1$, the mapping
  \begin{align*}
    \theta \mapsto \theta \int_{\bR^{d}} u(x)^{2} \mu(dx) - \cE^{\alpha}(u,u)
  \end{align*}
  is an affine function of $\theta$. 
  It is a fundamental fact in convex analysis that the supremum of any family of affine functions is convex. 
  Therefore, $C(\theta)$ is a convex function on $\bR$.
  
  Since $\mu$ is a smooth measure belonging to the appropriate class
  (such as $\mathcal{K}$ or $\mathcal{K}_{\infty}$),
  the operator is bounded from below,
  ensuring that $C(\theta)$ is finite for all $\theta \in \bR$. 
  Because any real-valued convex function defined on an open interval
  (in this case, $\bR$) is automatically continuous,
  the continuity of $C(\theta)$ follows immediately.

  Finally, we show the non-negativity of $C(\theta)$.
  We can choose a sequence of $L^{2}$-normalized test functions
  $\{u_{n}\} \subset \cF^{\alpha}$ that flatten out by spatial scaling
  (e.g., $u_{n}(x) = n^{-d/2}\phi(x/n)$ for some $\phi \in C_{0}^{\infty}(\bR^{d})$
  with $\|\phi\|_{2} = 1$). 
  For such a sequence, the Dirichlet form $\cE^{\alpha}(u_{n}, u_{n})$
  vanishes as $n \to \infty$ because the symbol $\Psi(\xi) = \log(1+|\xi|^{\alpha})$
  vanishes at the origin. 
  Furthermore, the integral $\int_{\bR^{d}} u_{n}(x)^{2} \mu(dx)$
  also vanishes as $n \to \infty$
  because $\mu$ is in the class $\mathcal{K}_{\infty}$. 
  This implies that the supremum in \eqref{eq:spectral_sup}
  is bounded from below by 0. 
  Consequently, we obtain $C(\theta) \ge 0$ for all $\theta \in \bR$.
\end{proof}

We define
\begin{align}
\label{eq:41}
\theta_{c} = \inf\{ \theta > 0 : C(\theta) > 0\}.
\end{align}
This constant $\theta_{c}$ is referred to as the \textit{critical coupling constant}.

  \section{The critical coupling constant and ground state}
  First, we consider the following variational problem:
  \begin{align}\label{eq:18}\
    \lambda(\theta) = \inf \left\{ \cE^{\alpha}(u,u)
    : u \in \cF_{e}^{\alpha},\ \theta \int_{\bR^{d}} u^{2} d\mu = 1 \right\}.
  \end{align}
  It can be shown that a minimizer $h$ for \eqref{eq:18} exists.
  We call the minimizer $h$ the \textit{ground state}.
  Then $h$ is the unique, bounded,
  continuous, and strictly positive function on $\bR^{d}$.
  Following the arguments in \cite{tsuchida2026GSPSF},
  we can construct the transition density of $\bfG^{\alpha}$.

  The following theorem plays a crucial role throughout this paper. 

  \begin{theorem}
  \label{thm:cpt}
  If $\mu \in \mathcal{K}_{\infty}$, there exists the ground state for \eqref{eq:18}. 
  Moreover, the function $h$ is uniquely determined and
  we can take a strictly positive, bounded and continuous version of $h$. 
\end{theorem}
\begin{proof}
In \cite{tsuchida2026GSPSF}, we show that $\bfG^{\alpha}$ has the 
irreducibility and the strong Feller property. 
By \cite[Lemma 2.4]{Takeda2014Variational},
there exists a strictly positive continuous function $g$ such that
$g \cdot dx \in \cK_{\infty}$. Put $\nu = \theta_{c} \mu + g \cdot dx$. 
Since the measure $\nu$ belongs to $\cK_{\infty}$
and has full topological support,  
the time-changed process of $\bfG^{\alpha}$ with $\nu$ belongs to
Class (T) introduced in \cite{Takeda2019}. 
Hence by \cite[Theorem 3.3, Theorem 5.4]{Takeda2019}, 
there exists a minimizer $h_{1}$ in the variational problem \eqref{eq:18}. 
Moreover, we can take $h_{1}$ as a strictly positive,
bounded and continuous function. 
Applying \cite[Lemma 2.5, Theorem 2.1]{Takeda2014Variational} and
putting $h = h_{1}/\sqrt{a}$ where 
$a = \theta\int_{\bR^{d}}^{ } h_{1}^{2} d\mu$,
the function $h$ is desired one. 
\end{proof}

\begin{lemma}[{\cite[Lemma 2.2]{TakedaTsuchida2007}}]
  \label{lem:equiv}
  The following statements are equivalent:
  \begin{enumerate}[{\rm (i)}]
  \item $\ds \inf \left\{ \cE^{\alpha}(u,u) : \int_{\bR^{d}}^{ } u^{2} d\mu = 1 \right\} < 1$,
  \item $\ds \inf \left\{ \cE^{\alpha}(u,u) - \int_{\bR^{d}}^{ } u^{2} d\mu : \int_{\bR^{d}}^{ } u^{2} dx = 1\right\} < 0$.
  \end{enumerate}
\end{lemma}

Before characterizing the critical coupling constant $\theta_c$,
let us introduce the notion of subcriticality and its analytic test.
For a smooth measure $\nu$, a Schr\"odinger-type operator $\mathcal{H}^{\nu} = -\mathcal{H}^{\alpha} - \nu$
is said to be \textit{subcritical} if it possesses
the minimal positive Green function $G^{\nu}(x, y)$.
The following analytic characterization (subcriticality test)
in terms of the principal eigenvalue of the time-changed process was initially established by
Takeda \cite[Theorem 3.9]{Takeda2002Conditional} under
some regularity conditions, and shortly after,
proven by Chen \cite{Chen:2002:Gaugeability}
for general symmetric Markov processes without regularity conditions.

\begin{theorem}[{\cite[Theorem 5.3]{Chen:2002:Gaugeability}}]
  \label{thm:subcritical_test}
  Let $\mu$ be a positive smooth measure in $\cS_\infty (= \mathcal{K}_{\infty})$. 
  Then the operator $\cH^{\mu}$ is subcritical if and only if
\begin{equation}
    \inf \left\{ \cE^{\alpha}(u, u) : u \in \mathcal{F}_e^\alpha, \int_{\mathbb{R}^d} u^2 d\nu = 1 \right\} > 1.
\end{equation}
\end{theorem}

Furthermore, assuming that $\cH^{\nu}$ is subcritical,
by the conditional gauge theorem established
for general symmetric Markov processes
(see Chen \cite[Theorem 3.6 and Remark after Theorem 3.10]{Chen:2002:Gaugeability}),
this subcriticality ensures that the perturbed Green function $G^{\nu}(x, y)$ exists
and is comparable to the original Green function $G(x, y)$, that is, 
there exists a constant $C > 1$ such that
\begin{equation}
  \label{eq:comparable}
    C^{-1} G(x, y) \le G^{\nu}(x, y) \le C G(x, y) \quad \text{for } x \neq y.
\end{equation}

\begin{lemma}
  \label{lem:char}
  The critical coupling constant $\theta_{c}$ is characterized
  as the unique positive number such that
  \begin{align}
  \label{eq:22}
    \inf \left\{ \cE^{\alpha}(u,u) : u \in \cF_{e}^{\alpha}, \ \theta_{c} \int_{\bR^{d}} u^{2} d\mu = 1 \right\} = 1.
  \end{align}
\end{lemma}
\begin{proof}
 Let us define a constant $A$ by the infimum of the Dirichlet form subject to the constraint $\int_{\bR^{d}} u^{2} d\mu = 1$:
  \begin{align}
  \label{eq:28}
    A = \inf \left\{ \cE^{\alpha}(u,u) : u \in \cF_{e}^{\alpha}, \ \int_{\bR^{d}} u^{2} d\mu = 1 \right\}.
  \end{align}
  First, we show that $A > 0$. By Theorem \ref{thm:SV}, we have
  \begin{align}
  \label{eq:75}
  \int_{\bR^{d}} u^{2} \, d\mu \le \| G \mu \|_{\infty} \cE^{\alpha}(u,u) \quad \text{for all } u \in \cF_{e}^{\alpha}. 
  \end{align}
  Taking the infimum over all $u \in \cF_{e}^{\alpha}$ with $\int_{\bR^{d}} u^{2} \, d\mu = 1$,
  we obtain $1 \le \| G \mu \|_{\infty} A$.
  Since $\mu \in \mathcal{K}_{\infty}$ implies that the potential $G \mu$ is bounded by \eqref{eq:SV-2},
  it follows that $A \ge \| G \mu \|_{\infty}^{-1} > 0$.

  For any given $\theta > 0$, we observe the following scaling property by taking $v = \sqrt{\theta}u$:
  \begin{align}
  \label{eq:25}
    \inf \left\{ \cE^{\alpha}(u,u) : \theta \int_{\bR^{d}} u^{2} d\mu = 1 \right\} 
    &= \inf \left\{ \cE^{\alpha}(v/\sqrt{\theta}, v/\sqrt{\theta}) : \int_{\bR^{d}} v^{2} d\mu = 1 \right\} \nonumber \\
    &= \frac{1}{\theta} \inf \left\{ \cE^{\alpha}(v,v) : \int_{\bR^{d}} v^{2} d\mu = 1 \right\} = \frac{A}{\theta}.
  \end{align}
  As established in Lemma \ref{lem:equiv}, the condition $C(\theta) > 0$ is equivalent to $\theta > A$. 
  Therefore, we obtain
  \begin{align}
  \label{eq:27}
    \theta_{c} = \inf\{\theta > 0 : C(\theta) > 0\} = \inf \left\{ \theta > 0 : \theta > A \right\} = A.
  \end{align}
  Since we have already shown $A > 0$, it follows that $\theta_{c}$ is strictly positive. 
  Consequently, substituting $\theta_{c} = A$ into \eqref{eq:25}, we have
  \begin{align}
  \label{eq:29}
    \inf \left\{ \cE^{\alpha}(u,u) : \theta_{c} \int_{\bR^{d}} u^{2} d\mu = 1 \right\} = \frac{A}{\theta_{c}} = 1.
  \end{align}
  This establishes the desired characterization and completes the proof.
\end{proof}

By Theorem \ref{thm:subcritical_test}, we immediately know that
$\mathcal{H}^{\theta_{c} \mu}$ is not subcritical.

\begin{lemma}[{\cite[Lemma 4.3]{TakedaTsuchida2007}}]
  \label{lem:Gnu}
  If $\phi \in C_{0}(\mathbb{R}^{d})$ is non-negative, then
  \begin{align}
  \label{eq:16}
    G^{\nu} \phi(\cdot) = \int_{\mathbb{R}^{d}}^{ } G^{\nu}(\cdot,y) \phi(y) dy \in \mathcal{F}_{e}^{\alpha}.
  \end{align}
\end{lemma}

\begin{lemma}\label{lem:subcritical_H_nu}
  For a non-negative function $w \in C_0(\mathbb{R}^d)$ with $w \not\equiv 0$,
  define a signed measure $\nu = -\theta_c \mu + w \cdot dx$. 
  Then the Schr\"odinger-type operator
  $\mathcal{H}^{\nu} := - \mathcal{H} + \nu$ is subcritical.
\end{lemma}
\begin{proof}
  According to the subcriticality test extended to signed measures
  (see, e.g., \cite{Chen:2002:Gaugeability}), $\mathcal{H}^{\nu}$ is subcritical
  if and only if
  \begin{align}
    \label{eq:subcritical_test_nu}
    \inf \left\{ \mathcal{E}^{\alpha}(u, u) + \int_{\mathbb{R}^d} u^2 w dx
      : u \in \mathcal{F}_e^\alpha, \
      \theta_c \int_{\mathbb{R}^d} u^2 d\mu = 1 \right\} > 1.
  \end{align}
  By the characterization of the critical coupling constant $\theta_c$
  in Lemma \ref{lem:char}, we know that
  $\mathcal{E}^{\alpha}(u, u) \ge 1$ for any $u \in \mathcal{F}_e^\alpha$
  satisfying $\theta_c \int_{\mathbb{R}^d} u^2 d\mu = 1$.
  Since $w \ge 0$, the infimum in \eqref{eq:subcritical_test_nu} is at least $1$.

  Suppose that the infimum in \eqref{eq:subcritical_test_nu} is exactly $1$.
  Then there exists a minimizing sequence
  $\{u_n\} \subset \mathcal{F}_e^\alpha$ such that
  $\theta_c \int_{\mathbb{R}^d} u_n^2 d\mu = 1$ and
  \begin{align}
  \label{eq:37}
    \lim_{n \to \infty} \left( \mathcal{E}^{\alpha}(u_n, u_n)
    + \int_{\mathbb{R}^d} u_n^2 w dx \right) = 1.
  \end{align}
  Since $\mathcal{E}^{\alpha}(u_n, u_n) \ge 1$ and
  $\int_{\mathbb{R}^d} u_n^2 w dx \ge 0$, it follows that
  \begin{align}
  \label{eq:44}
    \lim_{n \to \infty} \mathcal{E}^{\alpha}(u_n, u_n) = 1
    \quad \text{and} \quad \lim_{n \to \infty} \int_{\mathbb{R}^d} u_n^2 w dx = 0.
  \end{align}
  The first limit implies that $\{u_n\}$ is a minimizing sequence
  for the critical variational problem in Lemma 3.4. 
  By the compact embedding of $\mathcal{F}_e^\alpha$ into
  $L^2_{loc}(\mathbb{R}^d, dx)$ and into $L^2(\mu)$,
  there exists a subsequence (still denoted by $\{u_n\}$)
  that converges weakly in $\mathcal{F}_e^{\alpha}$, strongly in $L^2(\mu)$,
  and $dx$-almost everywhere to a minimizer $h_0 \in \mathcal{F}_e^{\alpha}$.

  By Theorem 3.1, the minimizer of this critical variational problem 
  \eqref{eq:22} is unique
  up to a sign, and we can choose a strictly positive continuous ground state $h$.
  Thus, we have $h_0 = h$ or $h_0 = -h$.
  Since $u_n(x) \to h_0(x)$ a.e. and $w \ge 0$,
  we can apply Fatou's lemma to the second limit in \eqref{eq:44} to obtain
  \begin{align}
  \label{eq:59}
    \int_{\mathbb{R}^d} h^2 w dx \le \liminf_{n \to \infty} \int_{\mathbb{R}^d} u_n^2 w dx = 0.
  \end{align}
  However, since $h$ is strictly positive everywhere and
  $w$ is a non-negative function with $w \not\equiv 0$,
  we strictly have $\int_{\mathbb{R}^d} h^2 w dx > 0$. This is a contradiction.

  Therefore, the infimum in \eqref{eq:subcritical_test_nu}
  must be strictly greater than $1$,
  which implies that the operator $\mathcal{H}^{\nu}$ is subcritical.
\end{proof}

Now, for some signed measure $\xi$, we set
\begin{align}
  \label{eq:42}
  \cF^{\xi} &= \cF^{\alpha} \cap L^{2}(|\xi|), \\
  \cE^{\xi}(u,v) &= \cE^{\alpha}(u,v) + \int_{\bR^{d}} u(x)v(x) \xi(dx) \quad
                   \text{for any }u,v \in \cF^{\xi}.
\end{align}
For a signed measure $\eta = \eta^+ - \eta^- 
\in \mathcal{K}_{\infty} - \mathcal{K}_{\infty}$, 
we note that $\mathcal{F}^{|\eta|} = \mathcal{F}^{\alpha}$ by 
Theorem \ref{thm:SV} and Lemma \ref{lem:relation-K}. 

\begin{lemma}
  \label{lem:inteq-1}
  For a non-negative function $w \in C_{0}(\bR^{d})$ with $w \not\equiv 0$,
  let $\nu = -\theta_{c} \mu + w \cdot dx$. Then the ground state $h$ satisfies
  \begin{align}
  \label{eq:36}
    h(x) = \int_{\bR^{d}}^{ } G^{\nu}(x,y) h(y) w(y) dy, \quad \text{for all $x \in \mathbb{R}^{d}$}.
  \end{align}
\end{lemma}
\begin{proof}
Note that
\begin{align}
\label{eq:32}
  \eqref{eq:22} \iff \inf \left\{ \cE^{\alpha}(u,u) : u \in \mathcal{F}_{e}^{\alpha},\
  \int_{\bR^{d}}^{ }u^{2} d\mu = 1\right\}
  = \theta_{c}.
\end{align}
By Theorem \ref{thm:cpt},
we can find a strictly positive bounded continuous ground state in \eqref{eq:22}.
Moreover, this ground state is unique. Let $h \in \mathcal{F}_{e}^{\alpha}$ be the ground state, that is,
\begin{align}
\label{eq:30}
  \cE^{\alpha}(h,h) = \theta_{c} \quad \text{\&} \quad \int_{\bR^{d}}^{ } h^{2} d\mu = 1.
\end{align}
By the Euler-Lagrange equation, it holds that for any $f \in \cF_{e}^{\alpha}$, 
\begin{align}
\label{eq:33}
  \cE^{\alpha}(h,f) = \theta_{c} \int_{\bR^{d}}^{ } hf d\mu
  \iff \cE^{-\theta_{c} \mu}(h,f) = 0. 
\end{align}
Then we see that
\begin{align}
\label{eq:43}
  \cE^{\nu}(h,f) = \int_{\bR^{d}}^{ } h(x) f(x) w(x) dx.
\end{align}
Since $\cH^{\nu} = -\cH^{\alpha} + \nu$ is subcritical by Lemma \ref{lem:subcritical_H_nu}, 
there exists the Green function $G^{\nu}(x,y)$.
Since $G^{\nu}\phi(x)$ is in $\mathcal{F}_{e}^{\alpha}$ for any non-negative function $\phi \in C_{0}(E)$
by Lemma \ref{lem:Gnu}, we substitute $G^{\nu} \phi$ as $f$ in \eqref{eq:43},
\begin{align}
\label{eq:45}
  & \cE^{\nu}(h,G^{\nu} \phi) = \int_{\bR^{d}}^{ } h(x) (G^{\nu} \phi(x)) w(x) dx 
  \iff \int_{\bR^{d}}^{ } h(x) \phi(x) dx = \int_{\bR^{d}}^{ } G^{\nu}(h w)(x) \phi(x) dx. 
\end{align}
Hence we have
\begin{align}
\label{eq:46}
  h(x) &= \int_{\bR^{d}} G^{\nu}(x,y) h(y)w(y) dy \nonumber \\
  &= \bE_{x}\left[\int_{0}^{\infty}
  \exp\left(\theta_{c} A_{t}^{\nu} - \int_0^{ t} w(X_{s}) ds \right) h(X_{t}) w(X_{t}) dt\right]
  \quad \text{$m$-a.e. }x.
\end{align}
We denote the right hand side as $v(x)$. 
First, using the same argument in \cite[Lemma 4.6]{TakedaTsuchida2007},
we find that $v$ is finely continuous. 
Since $h$ is continuous and $v$ is finely continuous,
$U := \{x \in \bR^{d} : h(x) \not= v(x)\}$ is finely open. 
Assume that there exists a point $x_{0} \in U$. Since $U$ is finely open,
it holds that for any $\alpha > 0$, $R_{\alpha}1_{U}(x_{0}) > 0$,
where $R_{\alpha}$ is the $\alpha$-resolvent of $\bfG^{\alpha}$. 
On the other hand, noting the absolute continuity of $\bfG^{\alpha}$,
\begin{align}
\label{eq:67}
  P_{t}(x_{0}, U) = \int_{U}^{ } p_{t}(x_{0},y) dy = 0, \quad t > 0,
\end{align}
since $U$ is of Lebesgue measure zero.
Hence it follows that
\begin{align}
\label{eq:68}
  R_{\alpha} 1_{U}(x_{0}) = \int_0^\infty e^{-\alpha t} P_t(x_0, U) dt = 0.
\end{align}
This is a contradiction. Therefore $h(x) = v(x)$ for all $x \in \bR^{d}$.
\end{proof}

\begin{lemma}
  \label{lem:asymp}
  There exists a constant $C > 1$ such that $h$ satisfies
  \begin{align}
  \label{eq:60}
    C^{-1}|x|^{\alpha - d} \le h(x) \le C |x|^{\alpha - d}, \quad |x| \to \infty.
  \end{align}
\end{lemma}
\begin{proof}
  Suppose that $0 \in \supp (w) \subset B(R)$. By Lemma \ref{lem:inteq-1} and
  the continuity of $h$, there exists a constant $C_{1} > 1$ such that 
  \begin{align}
  \label{eq:61}
    C_{1}^{-1} \int_{B(R)}^{ } G^{\nu}(x,y) w(y) dy
    \le h(x) \le C_{1} \int_{B(R)}^{ } G^{\nu}(x,y) w(y) dy.
  \end{align}
On account of the subcriticality of $\cH^{\nu}$,
there exists a constant $C_{2} > 1$ such that 
\begin{align}
\label{eq:48}
C_{2}^{-1}G(x,y) \le G^{\nu}(x,y) \le C_{2} G(x,y).
\end{align}
Thus we know from \eqref{eq:comparable} that for a constant $C_{3} > 1$, 
\begin{align}
\label{eq:62}
  C_{3}^{-1} \int_{B(R)}^{ } G(x,y) w(y) dy \le h(y) \le C_{3} \int_{B(R)}^{ } G(x,y) w(y) dy.
\end{align}
Note that for any fixed $x \in B(R)^{c}$, the function $y \mapsto G(x,y)$ is
harmonic with respect to $\bfG^{\alpha}$ in the domain $B(R)$. 
Applying the Harnack inequality (see \cite[Theorem 6.7]{SikicSongVondracek2006})
to $\{G(x,\cdot)\}_{x\in B(R)^{c}}$, there exists a constant $C_{4} > 1$ such that
for any $x \in B(R)^{c}$ and $y \in \supp (w)$,
\begin{align}
\label{eq:63}
C_{4}^{-1} G(x,y) \le G(x,0) \le C_{4} G(x,y).
\end{align}
Hence there exists a constant $C_{5} > 1$ such that for any $x \in B(R)^{c}$, 
\begin{align}
\label{eq:64}
C_{5}^{-1} G(x,0) \le h(x) \le C_{5} G(x,0).
\end{align}
Using Theorem \ref{thm:Green}, the proof is completed. 
\end{proof}

\begin{lemma}
The function $h$ is $P_{t}^{\theta_{c} \mu}$-excessive. 
\end{lemma}
\begin{proof}
  By a version of the Fukushima decomposition \cite[Theorem 5.2.2, Corollary 5.4.1]{FOT}
  for $h \in \mathcal{F}_{e}^{\alpha}$, we know that
\begin{align}
\label{eq:34}
  h(X_{t}) - h(X_{0}) = M_{t}^{h} - \theta_{c} \int_0^{ t} h(X_{s}) dA_{s}^{\mu},
  \quad \bP_{x}\text{-a.s.}\ \text{q.e. }x.
\end{align}
Noting the process $e^{\theta_{c}A_{t}^{\mu}}$ is continuous and of bounded
variation, Ito's formula yields 
\begin{align}
\label{eq:35}
  \exp(\theta_{c} A_{t}^{\mu}) h(X_{t})
  &= h(X_{0}) + \int_{0}^{t} \exp(\theta_{c} A_{s}^{\mu}) d(h(X_{s}))
    + \int_{0}^{t} h(X_{s}) d(\exp(\theta_{c} A_{s}^{\mu})) \\
  &= h(X_{0}) +
    \int_0^{ t} \exp(\theta_{c} A_{s}^{\mu}) dM_{s}^{h}
    - \theta_{c} \int_0^{ t} \exp(\theta_{c} A_{s}^{\mu}) h(X_{s})
    dA_{s}^{\mu} \\
  &\hspace{1cm} + \theta_{c} \int_0^{ t} h(X_{s})
    \exp(\theta_{c}A_{s}^{\mu}) dA_{s}^{\mu} \\
  &= h(X_{0}) +
    \int_0^{ t} \exp(\theta_{c} A_{s}^{\mu}) dM_{s}^{h}. 
\end{align}
Hence by the positivity of $h$, this implies
\begin{align}
\label{eq:23}
  \bE_{x}[\exp(\theta_{c} A_{t}^{\mu})h(X_{t})] \le h(x)
  \iff P_{t}^{\theta_{c} \mu} h(x) \le h(x).
\end{align}
\end{proof}

Since the function $h$ is $P_{t}^{\theta_{c}\mu}$-excessive,
we can consider the $h$-transformed process $\bfM^{h} = (\bP_{x}^{h}, X_{t}^{h})$, 
that is, the transition semigroup $\{P_{t}^{h}\}$ of $\bfM^{h}$ is defined by
\begin{align}
\label{eq:40}
  P_{t}^{h} f(x) &:= \frac{1}{h(x)} P_{t}^{\theta_{c} \mu} (h f)(x) \nonumber \\
                 &= \frac{1}{h(x)} \int_{\bR^{d}} p^{\theta_{c}\mu}(t,x,y) h(y) f(y) dy. 
\end{align}
In the remainder of this paper, since we only consider the $h$-transform
associated with $P_t^{\theta_c \mu}$, we omit $\theta_c \mu$ from the notation 
and simply write $P_t^h$.

It is known that the transition semigroup $P_{t}^{h}$ is $h^{2}m$-symmetric. 
With respect to this measure, $P_{t}^{h}$ has the transition density 
$p^{\theta_{c}\mu}(t,x,y)/(h(x)h(y))$.
Combining Theorem \ref{thm:subcritical_test} and Lemma \ref{lem:char},
$\cH^{\theta_{c} \mu}$ is not subcritical, 
which means that its minimal Green function does not exist, 
and consequently, that of the $h$-transformed process does not exist either.
However, since $P_{t}^{h}$ constitutes an $h^{2}m$-symmetric Markov semigroup, 
the non-existence of the Green function implies that $\bfM^{h}$ is recurrent.

A recurrent Markov process is called null recurrent (resp. positive recurrent) 
if the total mass of its symmetrized measure is infinite (resp. finite). 
Combining this definition with Lemma \ref{lem:asymp}, we obtain the following proposition.

\begin{proposition}
  \label{prop:null}
  If $\alpha < d \le 2 \alpha$, the process $\bfM^{h}$ is null recurrent.
  If $d > 2\alpha$, the process $\bfM^{h}$ is positive recurrent. 
\end{proposition}

If $\bfM^{h}$ is null recurrent (resp. positive recurrent), we say that the operator $\cH^{\theta_{c}\mu}$
is \textit{null critical} (resp. \textit{positive critical}).

\section{Differentiability of spectral functions}
At the begining of this section, we show the Harris recurrence of
$h$-transformed process $\bfM^{h}$ without the strong Feller property. 
\begin{proposition}
  \la{prop:Harris}
  The $h$-transformed process $\bfM^{h} = (\bP_{x}^{h}, X_{t}^{h})$ is Harris recurrent,
  that is, for a non-negative function $f$,
  \begin{align}
  \label{eq:24}
    \int_{0}^{\infty} f(X_{t}^{h}) dt = \infty \quad \bP_{x}^{h}\text{-a.s.}\
    \text{for all $x \in \bR^{d}$}
  \end{align}
  whenever $m(\{x : f(x) > 0\}) > 0$. 
\end{proposition}
\begin{proof}
  Let $B_t = \int_0^t f(X_s^h) ds$. 
  Note that $\{B_t\}$ is a positive continuous additive functional. 
  We define the function $u(x)$ by
  $$u(x) = \mathbb{E}_x^h \left[ e^{-B_\infty} \right].$$
  Our goal is to show that $u(x) = 0$ for all $x \in \bR^d$. 
  By the additive property $B_\infty = B_t + B_\infty \circ \theta_t$ and the Markov property,
  we have for any $t > 0$,
  \begin{align}\label{eq:u_markov}
    u(x) = \mathbb{E}_x^h \left[ e^{-B_t} e^{-B_\infty \circ \theta_t} \right]
    = \mathbb{E}_x^h \left[ e^{-B_t} u(X_t^h) \right] \quad \text{for all } x \in \bR^d.
  \end{align}
  Since $f \ge 0$, it is clear that $0 \le e^{-B_t} \le 1$, which trivially implies $0 \le u(x) \le 1$. 
  Equation \eqref{eq:u_markov} then yields
  \begin{align}\label{eq:u_ineq}
    u(x) \le \mathbb{E}_x^h [ u(X_t^h) ] = P_t^h u(x) \quad \text{for all } x \in \bR^d.
  \end{align}
  Recall that the $h$-transformed process is recurrent, hence, conservative;
  that is, its transition semigroup is Markovian and satisfies $P_t^h 1 = 1$. 
  Using this property, we consider the non-negative function $v(x) = 1 - u(x)$.
  It follows from \eqref{eq:u_ineq} that
  \begin{align}\label{eq:v_excessive}
    v(x) = 1 - u(x) \ge 1 - P_t^h u(x) = P_t^h (1 - u)(x) = P_t^h v(x).
  \end{align}
  This means that $v$ is a $P_t^h$-excessive function. 

  Note that the strict positivity of the transition density $p_{t}^{h}(x,y)$ 
  guarantees that the $h$-transformed process is irreducible. 
  Since the associated Dirichlet form $\cE^h$ is both irreducible and recurrent, 
  any $P_t^h$-excessive function must be constant q.e. by \cite[Theorem 4.7.1(iii)]{FOT}.
  Therefore, $v(x)$ is constant q.e., which implies that
  there exists a constant $C \in [0, 1]$ such that
  \begin{align}
  \label{eq:76}
  u(x) = C \quad \text{for q.e. } x.
  \end{align}
  
  Next, we determine the value of $C$.
  Substituting $u(x) = C$ back into \eqref{eq:u_markov} for those q.e. $x$,
  we obtain
  \begin{align}
  \label{eq:77}
    C = \mathbb{E}_x^h \left[ e^{-B_t} C \right] = C \cdot \mathbb{E}_x^h \left[ e^{-B_t} \right]
    \quad \text{for q.e. $x$}
  \end{align}
  By assumption, the set $\{y \in \bR^d : f(y) > 0\}$
  has a strictly positive measure $m$. 
  Since the transition density satisfies $p_t^h(x, y) > 0$ for all $t > 0$
  and $x, y \in \bR^d$, the process starting from
  any $x$ visits this set with a strictly positive probability. 
  Thus, we have $P_x^h(B_t > 0) > 0$, which yields
  \begin{align}
  \label{eq:78}
  \mathbb{E}_x^h \left[ e^{-B_t} \right] < 1.
  \end{align}
  Returning to the equation $C = C \cdot \mathbb{E}_x^h \left[ e^{-B_t} \right]$,
  since the expectation is strictly less than $1$,
  the only possible algebraic solution is $C = 0$. 
  Therefore, we obtain
  \begin{align}
  \label{eq:79}
  u(y) = 0 \quad \text{for q.e. } y.
  \end{align}
  Since $u(y) = 0$ a.e. $y$, it follows from \eqref{eq:u_ineq} that 
  \begin{align}
  \label{eq:80}
  u(x) \le P_t^h u(x) = \int_{\bR^d} p_t^h(x, y) u(y) dy = 0 \quad \text{for all $x$}.
  \end{align}
  Since $u \ge 0$ by definition, we conclude that
  \begin{align}
  \label{eq:81}
  u(x) = 0 \quad \text{for all } x \in \bR^d,
  \end{align}
  which completes the proof.
\end{proof}

\begin{lemma}
  \label{lem:L2loc}
  Under the transience condition $d > \alpha$,
  any function $u \in \cF_{e}^{\alpha}$ belongs to $L^{2}_{\text{loc}}(dx)$.
\end{lemma}
\begin{proof}
  Let $u \in \mathcal{F}_e^\alpha$.
  By the definition of the extended Dirichlet space,
  there exists an approximating sequence $\{u_n\} \subset \mathcal{F}^\alpha$
  such that $\lim_{n \to \infty} u_n = u$ $m$-a.e. and
  $\mathcal{E}^\alpha(u_n - u_m, u_n - u_m) \to 0$ as $n, m \to \infty$. 
  Since $u_n \in L^2(\mathbb{R}^d)$, its Fourier transform $\hat{u}_n(\xi)$
  is well-defined, and by the Plancherel theorem,
  the energy form can be represented as
  \begin{align}
    \mathcal{E}^\alpha(u_n, u_n) \asymp
    \int_{\mathbb{R}^d} |\hat{u}_n(\xi)|^2 \Psi(\xi) d\xi,
  \end{align}
  where $\Psi(\xi) = \log(1 + |\xi|^\alpha)$ is the characteristic exponent of the GSP.
  We decompose $u_n$ into two parts, $u_n = v_n + w_n$,
  defined via the inverse Fourier transform by splitting the domain
  into $\{|\xi| \le 1\}$ and $\{|\xi| > 1\}$:
  \begin{align}
    \hat{v}_n(\xi) = \hat{u}_n(\xi) 1_{\{|\xi| \le 1\}} \quad
    \text{and} \quad \hat{w}_n(\xi) = \hat{u}_n(\xi) 1_{\{|\xi| > 1\}}.
  \end{align}

  For the part $w_n$,
  since $\Psi(\xi)$ is bounded away from zero on
  $\{|\xi| > 1\}$ (i.e., $\Psi(\xi) \ge \log 2$), we have
  \begin{align}
    \int_{\mathbb{R}^d} |\hat{w}_n(\xi) - \hat{w}_m(\xi)|^2 d\xi
    &\le \frac{1}{\log 2} \int_{\{|\xi| > 1\}} |\hat{u}_n(\xi) - \hat{u}_m(\xi)|^2 \Psi(\xi) d\xi \\
    &\le C \mathcal{E}^\alpha(u_n - u_m, u_n - u_m) \to 0.
  \end{align}
  Thus, $\{\hat{w}_n\}$ is a Cauchy sequence in $L^2(\mathbb{R}^d)$,
  which implies $\{w_n\}$ converges strongly in $L^2(\mathbb{R}^d)$.

  For the part $v_n$, since $\Psi(\xi) \asymp |\xi|^\alpha$ for $|\xi| \le 1$,
  we can apply the Schwarz inequality to obtain
  \begin{align}
    \int_{\mathbb{R}^d} |\hat{v}_n(\xi) - \hat{v}_m(\xi)| d\xi 
    &= \int_{\{|\xi| \le 1\}} |\hat{u}_n(\xi) - \hat{u}_m(\xi)| d\xi \\
    &\le \left( \int_{\{|\xi| \le 1\}} |\hat{u}_n(\xi) - \hat{u}_m(\xi)|^2
      |\xi|^\alpha d\xi \right)^{1/2} \left( \int_{\{|\xi| \le 1\}} |\xi|^{-\alpha} d\xi \right)^{1/2} \\
    &\le C' \mathcal{E}^\alpha(u_n - u_m, u_n - u_m)^{1/2}
      \left( \int_{\{|\xi| \le 1\}} |\xi|^{-\alpha} d\xi \right)^{1/2}.
  \end{align}
  Under the transience condition $d > \alpha$,
  the integral $\int_{\{|\xi| \le 1\}} |\xi|^{-\alpha} d\xi$ is finite.
  Therefore, $\{\hat{v}_n\}$ is a Cauchy sequence in $L^1(\mathbb{R}^d)$.
  This implies that $\{v_n\}$ is a Cauchy sequence in $L^\infty(\mathbb{R}^d)$
  and thus converges uniformly.

  Since $\{w_n\}$ converges in $L^2(\mathbb{R}^d)$ and $\{v_n\}$
  converges in $L^\infty(\mathbb{R}^d)$,
  their sum $u_n = v_n + w_n$ is a Cauchy sequence
  in $L^2_{\loc}(\mathbb{R}^d)$ and thus converges to some
  $\tilde{u} \in L^2_{\loc}(\mathbb{R}^d)$.
  Furthermore, since $L^2_{\loc}$-convergence implies pointwise convergence
  $m$-a.e. along a subsequence and $u_n \to u$ $m$-a.e. by definition,
  the limits must coincide (i.e., $\tilde{u} = u$ $m$-a.e.).
  Therefore, the limit $u$ belongs to $L^2_{\loc}(\mathbb{R}^d)$.
\end{proof}

\begin{lemma}
  \label{lem:green_potential_bound}
  Let $\psi \in C_0^\infty(\mathbb{R}^d)$ be a non-negative function,
  and define $K_\psi(y) = \int_{\mathbb{R}^d} (\psi(x)-\psi(y))^2 J(x,y)dx$.
  Then the Green potential of $K_\psi(y)dy$ is uniformly bounded
  on $\mathbb{R}^d$, i.e.,
\begin{equation*}
    \sup_{x\in\mathbb{R}^d} \int_{\mathbb{R}^d} G(x, y)K_\psi(y)dy < \infty.
\end{equation*}
\end{lemma}
\begin{proof}
  Since $\psi \in C_0^\infty(\mathbb{R}^d)$ is globally Lipschitz and bounded,
  we have $(\psi(x)-\psi(y))^2 \le C(1 \wedge |x-y|^2)$.
  The property of the L\'evy measure
  $\int_{\mathbb{R}^d} (1 \wedge |x-y|^2)J(x,y)dx < \infty$ ensures that
  $K_\psi(y)$ is bounded by some constant $M > 0$ on $\mathbb{R}^d$.
  Furthermore, since $|x-y| \ge |y|/2$ for $x \in \supp(\psi)$ and
  sufficiently large $|y|$, it follows directly from
  Theorem \ref{thm:Levy} that
  \begin{align}
  \label{eq:74}
    K_{\phi}(y) &= \int_{\mathbb{R}^{d}}^{ } \psi(x)^{2} J(x,y) dx
                  \le C \int_{\supp(\psi)}^{ } J(x,y) dx \\
                  &\le C' \int_{\supp(\psi)}^{ } J(x,y) dx \le C'' |y|^{-d-\alpha}.
  \end{align}
  (Note that when $\alpha=2$,
  the jump kernel $J(x,y)$ and this integral actually decay exponentially as
  $|y| \to \infty$, but the polynomial bound $C|y|^{-d-\alpha}$ is sufficient
  for our upper estimate for all $\alpha \in (0, 2]$). 
  Hence we see that $K_\psi(y)$ is bounded on $\mathbb{R}^d$
  and decays at least at the order of $|y|^{-d-\alpha}$ at infinity.
  Let $0 < \varepsilon < 1$.
  We divide the domain of integration into 
  $\{|x-y| < \varepsilon\}$ and $\{|x-y| \ge \varepsilon\}$. 

  \begin{enumerate}[(i)]
  \item In the region $|x-y| < \varepsilon$:
    Using the uniform bound $K_\psi(y) \le M$, we have
    \begin{align}
    \label{eq:87}
    \int_{|x-y| < \varepsilon} G(x,y)K_\psi(y)dy \le M \int_{|x-y| < \varepsilon} G(x,y)dy.
    \end{align}
    By Theorem 2.1, $G(x,y) \sim C|x-y|^{-d}(\log(1/|x-y|))^{-2}$ near the origin.
    Since this singularity is locally integrable in $\mathbb{R}^d$,
    the integral converges to a finite constant independent of $x$.
  \item In the region $|x-y| \ge \varepsilon$:
    Note that in this region, $G(x,y)$ is continuous and bounded
    by $g(\varepsilon) < \infty$.
    Thus, the integral over any bounded intermediate region is trivially finite.
    If the starting point $x$ is confined to a bounded region,
    $|x-y| \sim |y|$ as $|y| \to \infty$.
    Thus $G(x,y)$ decays as $|x-y|^{\alpha-d} \le C|y|^{\alpha-d}$
    and $K_\psi(y)$ decays at least at the order of $|y|^{-d-\alpha}$ at infinity,
    making the integrand decay as $|y|^{-2d}$. Hence the integral converges.

    To ensure the uniform boundedness as $|x| \to \infty$,
    we further split the region $\{|x-y| \ge \varepsilon\}$
    into $\{|x-y| \le |x|/2\}$ and $\{|x-y| > |x|/2\}$.
    \begin{itemize}
    \item In the region $\{|x-y| \le |x|/2\}$: 
      We have $|y| \ge |x|/2$.
      Thus $K_\psi(y) \le C|y|^{-d-\alpha} \le C'|x|^{-d-\alpha}$.
      Factoring this out, the remaining integral of $G(x,y)$
      over this local region grows at most as $|x|^\alpha$.
      Multiplying these gives an overall bound of order $|x|^{-d}$,
      which decays to $0$ as $|x| \to \infty$.
    \item In the region $\{|x-y| > |x|/2\}$: 
      We have $G(x,y) \le C|x-y|^{\alpha-d} \le C'|x|^{\alpha-d}$.
      Since $K_\psi(y)$ is globally integrable ($\int_{\mathbb{R}^d} K_\psi(y)dy < \infty$),
      the integral over this region is bounded by $C''|x|^{\alpha-d}$.
      Because of the transience condition $d > \alpha$,
      this also decays to $0$ as $|x| \to \infty$.
    \end{itemize}
  \end{enumerate}
  Consequently, the Green potential is uniformly bounded across the entire space.
\end{proof}
\begin{lemma}
  \label{lem:multiplier}
  Let $\psi \in C_{0}^{\infty}(\bR^{d})$ be a non-negative function. 
  If $u \in \cF_{e}^{\alpha}$, then the pointwise product $u \psi$
  belongs to $\cF^{\alpha}$.
\end{lemma}
\begin{proof}
  First, we consider the case where $u \in \mathcal{F}^\alpha$.
  Since $u \in \mathcal{F}^\alpha \subset L^2(\mathbb{R}^d)$ and
  $\psi \in C_0^\infty(\mathbb{R}^d)$,
  it is clear that the product $u\psi \in L^2(\mathbb{R}^d)$.
  We evaluate the energy of $u\psi$.
  Noting the algebraic identity
  \begin{align}
    (u(x)\psi(x) - u(y)\psi(y))^2 \le 2\psi(x)^2(u(x)-u(y))^2 + 2u(y)^2(\psi(x)-\psi(y))^2,
  \end{align}
  and using the symmetry $J(x, y) = J(y, x)$, integrating this inequality yields
  \begin{align}
    \mathcal{E}^\alpha(u\psi, u\psi)
    &\le 2 \iint_{\mathbb{R}^d \times \mathbb{R}^d} \psi(x)^2(u(x)-u(y))^2 J(x,y) dx dy \\
    &\hspace{1cm}
      + 2 \iint_{\mathbb{R}^d \times \mathbb{R}^d} u(y)^2(\psi(x)-\psi(y))^2 J(x,y) dx dy.
  \end{align}
  For the first term, since $\psi$ is bounded,
  it is bounded by $2\|\psi\|_\infty^2 \mathcal{E}^\alpha(u, u)$. 
  For the second term,
  by using $K_{\psi}$ in Lemma 
  \ref{lem:green_potential_bound} 
  it can be rewritten as $2\int_{\mathbb{R}^d} u(y)^2 K_\psi(y) dy$. 
By Lemma \ref{lem:green_potential_bound}, we know that the Green potential of $K_\psi(y)dy$ is bounded 
on $\mathbb{R}^d$. Therefore, the transient Poincar\'e-type inequality 
(Theorem \ref{thm:SV}) guarantees that the second term is bounded by $C \mathcal{E}^\alpha(u, u)$ for some constant $C > 0$.
Combining these estimates, we obtain
\begin{align}
    \mathcal{E}^\alpha(u\psi, u\psi) \le (2\|\psi\|_\infty^2 + C)\mathcal{E}^\alpha(u, u) < \infty.
\end{align}
Since $u\psi \in L^2(\mathbb{R}^d)$ and its energy is finite, we have $u\psi \in \mathcal{F}^\alpha$. 

This result can be extended to $u \in \mathcal{F}_e^\alpha$. Indeed, by the definition of the extended Dirichlet space, there exists an approximating sequence $\{u_n\} \subset \mathcal{F}^\alpha$ such that $u_n \to u$ $m$-a.e. and $\mathcal{E}^\alpha(u_n - u_m, u_n - u_m) \to 0$ as $n, m \to \infty$. Applying the above inequality to the difference $u_n - u_m \in \mathcal{F}^\alpha$, we have
\begin{align}
    \mathcal{E}^\alpha(u_n\psi - u_m\psi, u_n\psi - u_m\psi) \le (2\|\psi\|_\infty^2 + C)\mathcal{E}^\alpha(u_n - u_m, u_n - u_m) \to 0.
\end{align}
This implies that $\{u_n\psi\}$ is an $\mathcal{E}^\alpha$-Cauchy sequence in $\mathcal{F}^\alpha$. Since $u_n\psi \to u\psi$ $m$-a.e., the limit $u\psi$ belongs to $\mathcal{F}_e^\alpha$ by definition. Furthermore, by Lemma 4.2, $u \in \mathcal{F}_e^\alpha$ implies $u \in L^2_{\loc}(\mathbb{R}^d)$. Since $\psi$ has compact support, it follows that $u\psi \in L^2(\mathbb{R}^d)$. Therefore, $u\psi \in \mathcal{F}_e^\alpha \cap L^2(\mathbb{R}^d) = \mathcal{F}^\alpha$.
\end{proof}

\begin{lemma}
  \label{lem:compact_embedding}
  Let $\cF_{e}^{\alpha}$ be the extended Dirichlet space associated with
  the geometric $\alpha$-stable process. 
  Under the transience condition $d > \alpha$,
  the space $\cF_{e}^{\alpha}$ is compactly embedded into $L^{2}_{\text{loc}}(\bR^{d}, dx)$. 
  That is, any sequence $\{u_{n}\}$ that is bounded in $\cF_{e}^{\alpha}$ has
  a subsequence that converges strongly in $L^{2}(K, dx)$ for any compact set $K \subset \bR^{d}$.
\end{lemma}
\begin{proof}
  Let $\{u_{n}\}$ be a bounded sequence in $\cF_{e}^{\alpha}$,
  i.e., $\sup_{n} \cE^{\alpha}(u_{n}, u_{n}) < \infty$. 
  To prove the local compactness, it suffices to show that for any compact set $K \subset \bR^{d}$,
  there exists a subsequence that converges strongly in $L^{2}(K, dx)$.
  
  Let $K$ be an arbitrary compact set.
  We take a smooth cutoff function $\psi \in C_{0}^{\infty}(\bR^{d})$ such that $0 \le \psi \le 1$
  and $\psi \equiv 1$ on $K$. 
  We define a localized sequence $v_{n} = u_{n} \psi$.
  To establish the local strong convergence of $\{u_{n}\}$,
  we first show that the localized sequence $v_{n} = u_{n} \psi$ is bounded
  in the Dirichlet space $\cF^{\alpha} = \cF_{e}^{\alpha} \cap L^{2}(\bR^{d})$.
  By Lemma \ref{lem:multiplier}, we have already known that the pointwise product
  $v_{n}$ belongs to $\cF_{e}^{\alpha}$, and its energy $\cE^{\alpha}(v_n, v_n)$ is uniformly bounded
  by the original energy of $u_n$. 
  Moreover, since $\psi$ has compact support and $u_{n} \in L^{2}_{\text{loc}}(dx)$
  by Lemma \ref{lem:L2loc}, the sequence $\{v_{n}\}$ is also bounded in $L^{2}(\bR^{d})$. 
  Thus, $\{v_{n}\}$ is $\mathcal{E}_{1}$-bounded in $\cF^{\alpha}$.

  Recall that the characteristic exponent of the geometric $\alpha$-stable process is
  $\Psi(\xi) = \log(1 + |\xi|^{\alpha})$, which satisfies $\lim_{|\xi| \to \infty} \Psi(\xi) = \infty$. 
  According to Rellich type compactness theorem
  for pseudodifferential operators (cf. \cite[Remark 3.10.6]{Jacob2001Vol1}),
  such a growth condition on the symbol $\Psi(\xi)$ ensures that any bounded sequence
  in $\cF^{\alpha}$ has a subsequence that converges strongly in $L^{2}_{\text{loc}}(\bR^{d}, dx)$.
  Since our sequence $\{v_{n}\}$ is bounded in $\cF^{\alpha}$,
  there exists a subsequence $\{v_{n_{k}}\}$ that converges strongly in $L^{2}_{\text{loc}}(\bR^{d}, dx)$. 
  Since $v_{n} = u_{n}$ on $K$, this implies that $\{u_{n_{k}}\}$ converges strongly in $L^{2}(K, dx)$. 
  By a diagonal argument over an increasing sequence of compact sets covering $\bR^{d}$,
  we conclude that $u_{n}$ converges strongly in $L^{2}_{\text{loc}}(\bR^{d}, dx)$.
\end{proof}

Next, we give an extension of Oshima's inequality
(\cite{Oshima1982Potential923}).
Before stating the next proposition, we state a technical remark.
Strictly speaking, to execute the limit operation in the proof of our main theorem, 
it would be sufficient to use an indicator function of a compact set, 
as in Oshima's original work. 
However, we establish the generalized inequality using a smooth cutoff function 
$\psi \in C_{0}^{\infty}(\bR^{d})$ in the following proposition. 
We adopt this formulation because such a function serves as a valid multiplier 
in the extended Dirichlet space, connecting potential theory and local energy estimates, which gives a stronger result.

\begin{proposition}
  \la{prop:Oshima-ineq}
  Let $h$ be the strictly positive continuous bounded ground state in \eqref{eq:22}. 
  There exists a positive function $g \in L^{1}(h^{2} dx)$ and a non-negative function
  $\psi \in C_{0}^{\infty}(\bR^{d})$ with $\int_{\bR^{d}}^{ } \psi h^{2} dx = 1$ such that
  for any $u \in \cF_{e}^{\alpha}$, 
  \begin{align}
  \label{eq:21}
    &\hspace{-1cm}
      \int_{\bR^{d}}^{ } \left| u(x) - h(x) L \left( \frac{u}{h} \right) \right| g(x) h(x) dx \\
    &\le C \mathcal{E}^{-\theta_{c} \mu}(u,u)^{1/2} 
    = C \left( \cE^{\alpha}(u,u) -
      \theta_{c} \int_{\bR^{d}}^{ } u(x)^{2} d\mu \right)^{1/2}
  \end{align}
  where $L(u)$ is a linear functional defined by 
  \begin{align}
  \label{eq:26}
    L(u) = \int_{\bR^{d}}^{ } u(x) \psi(x) h(x)^{2} dx.
  \end{align}
\end{proposition}
\begin{proof}
  Noting that $\mathbf{M}^h$ is the $h^2dx$-symmetric Harris recurrent
  Markov process, we first fix a measurable set $A$ with $0 < \int_A h(x)^2 dx < \infty$ and define a continuous additive functional $C_t$ by
\begin{align}
\label{eq:82}
  C_{t} &:= \int_0^{t} 1_{A}(X_{s}) ds.
\end{align}
  Let $R^{C}$ be the potential operator associated with $C_t$, defined by
\begin{align}
\label{eq:83}
  R^{C} f(x) &:= \mathbb{E}_{x} \left[ \int_0^{\infty} e^{-C_{t}} f(X_{t}) dt \right] \quad
               \text{for $f \in \mathcal{B}_{b}^{+}(\mathbb{R}^{d})$}. 
\end{align}
  According to Oshima's inequality (\cite{Oshima1982Potential923} or \cite[Theorem 4.8.2]{FOT}),
  there exist a strictly positive function $g_0 \in L^1(h^2dx)$ with $\|R^C g_0\|_\infty < \infty$ and
  a constant $C_0 > 0$ such that for any $v \in \mathcal{F}_{e}^{h}$,
\begin{align}\label{eq:oshima_g0}
  \int_{\mathbb{R}^d} |v(x) - K(v)| g_0(x) h(x)^2 dx &\le C_0 \mathcal{E}^h(v,v)^{1/2},
\end{align}
where $K(v) = \frac{1}{\int_A h(x)^2 dx}\int_A v(x)h(x)^2 dx$.

Let $\psi \in C_0^\infty(\mathbb{R}^d)$ be a non-negative function such that $\int_{\mathbb{R}^d} \psi(x) h(x)^2dx = 1$ and $\supp(\psi) \subset A$. 
Since $\psi(x) \le \|\psi\|_\infty 1_A(x)$ for all $x \in \mathbb{R}^d$, we have
\begin{align}
  R^C \psi(x) &\le \|\psi\|_\infty \mathbb{E}_x \left[ \int_0^\infty e^{-C_t} 1_A(X_t) dt \right] \notag \\
  &= \|\psi\|_\infty \mathbb{E}_x \left[ \int_0^\infty e^{-C_t} dC_t \right] \notag \\
  &= \|\psi\|_\infty \mathbb{E}_x \left[ 1 - e^{-C_\infty} \right] \le \|\psi\|_\infty. \notag
\end{align}
Then we see that $R^{C} \psi$ is bounded, i.e., $\|R^C \psi\|_\infty < \infty$. 
Thus, we may replace the initial weight function $g_0$ with $g := g_0 + \psi$, because the new weight function $g$ still satisfies the boundedness of the potential:
\begin{align}
  \|R^C g\|_\infty &\le \|R^C g_0\|_\infty + \|R^C \psi\|_\infty < \infty. \notag
\end{align}
Note that $g \in L^1(h^2dx)$ is strictly positive and $g(x) \ge \psi(x)$ for all $x \in \mathbb{R}^d$.
Consequently, by \cite[Corollary 4.8.1]{FOT}, there exists a constant $C_1 > 0$ such that
\begin{align}
  \label{eq:oshima_g}
  \int_{\mathbb{R}^d} |v(x) - K(v)| g(x) h(x)^2 dx &\le C_1 \mathcal{E}^h(v,v)^{1/2},
  \quad \text{for all } v \in \mathcal{F}_e^h.
\end{align}

Defining the linear functional $L(v) = \int_{\mathbb{R}^d} v(x)\psi(x)h(x)^2 dx$, we can evaluate the difference between $L(v)$ and $K(v)$
by noting that $K(v) = \int_{\mathbb{R}^d} K(v)\psi(x)h(x)^2 dx$.
Since $\psi(x) \le g(x)$, we obtain
\begin{align}
|L(v) - K(v)| &= \left| \int_{\mathbb{R}^d} (v(x) - K(v)) \psi(x) h(x)^2 dx \right| \notag \\
                &\le \int_{\mathbb{R}^d} |v(x) - K(v)| \psi(x) h(x)^2 dx \notag \\
  &\le \int_{\mathbb{R}^d} |v(x) - K(v)| g(x) h(x)^2 dx \le C_1 \mathcal{E}^h(v,v)^{1/2}. \notag
\end{align}

Combining this with \eqref{eq:oshima_g} and applying the triangle inequality,
we have
\begin{align}
  &\hspace{-0.5cm} \int_{\mathbb{R}^d} |v(x) - L(v)| g(x) h(x)^2 dx \notag \\
  &\le \int_{\mathbb{R}^d} |v(x) - K(v)| g(x) h(x)^2 dx
    + |K(v) - L(v)| \int_{\mathbb{R}^d} g(x) h(x)^2 dx \notag \\
  &\le C_1 \mathcal{E}^h(v,v)^{1/2}
    + C_1 \|g\|_{L^1(h^2dx)} \mathcal{E}^h(v,v)^{1/2} 
  = C_2 \mathcal{E}^h(v,v)^{1/2}, \notag
\end{align}
where $C_2 = C_1(1 + \|g\|_{L^1(h^2dx)})$. 
Substituting $v = u/h$ in this inequality together with (4.4),
we directly obtain
\begin{align}
\label{eq:31}
  \int_{\mathbb{R}^d} \left| u(x) - h(x)\left( \frac{u}{h} \right) \right| g(x) h(x) dx
  &\le C_2 \left( \mathcal{E}^{\alpha}(u,u)
  - \theta_{c} \int_{\mathbb{R}^{d}}^{ }u^{2} d\mu\right)^{1/2}, \quad u \in \mathcal{F}_{e}^{\alpha}.
\end{align}
This completes the proof.
\end{proof}

Now, we prove the main theorem in this paper.

\begin{theorem}
  \label{thm:main}
  Assume that $\mu$ belongs to $\cS_{\infty}$.
  If $\alpha < d \le 2\alpha$, $C(\theta)$ is differentiable
  with respect to $\theta \in \bR^{d}$. 
\end{theorem}
\begin{proof}
  For $\theta > \theta_{c}$, we see from the same argument of \cite[Lemma 4.3]{Takeda2003Large}
  that $-C(\theta)$ be the principal eigenvalue
  of $\cH^{\theta \mu} = -\cH - \theta \mu$ with finite multiplicity.
  Let $u_{\theta}$ be the $L^{2}(dx)$-normalized eigenfunction associated with $-C(\theta)$. 
  By the analytic perturbation theory, the $C(\theta)$ is differentiable, and its derivative is given by
  \begin{align}
  \label{eq:20}
    \frac{dC(\theta)}{d\theta} = \int_{\bR^{d}} u_{\theta}^{2} d\mu.
  \end{align}

  Let $\{\theta_{n}\}$ be a real sequence strictly decreasing to $\theta_{c}$,
  and let $u_{n}$ (written simply for $u_{\theta_{n}}$) be the $L^{2}(dx)$-normalized eigenfunction
  corresponding to $-C(\theta_{n})$. Then it holds that 
  \begin{align}
  \label{eq:49}
    -C(\theta_{n}) = \cE^{\alpha}(u_{n}, u_{n}) - \theta_{n} \int_{\bR^{d}} u_{n}^{2} d\mu
    = \mathcal{E}^{-\theta_{n} \mu}(u_{n}, u_{n}).
  \end{align}
  By Theorem \ref{thm:SV}, 
  for any $\varepsilon > 0$ there exists a constant $M(\varepsilon) > 0$ such that
  \begin{align}
  \label{eq:84}
    \int_{\mathbb{R}^{d}}^{ } u(x)^{2} \mu(dx) \le \varepsilon \mathcal{E}^{\alpha}(u,u) +
    M(\varepsilon) \int_{\mathbb{R}^{d}}^{ } u(x)^{2} dx.
  \end{align}
  Thus, we can estimate $\mathcal{E}^{\alpha}(u_{n}, u_{n})$ as follows:
  \begin{align}
  \label{eq:50}
    \cE^{\alpha}(u_{n},u_{n}) 
    &= -C(\theta_{n}) + \theta_{n} \int_{\bR^{d}} u_{n}^{2} d\mu \nonumber \\
    &\le -C(\theta_{n}) + \theta_{n} \varepsilon \cE^{\alpha}(u_{n},u_{n}) + \theta_{n} M(\varepsilon).
  \end{align}
  Taking $\varepsilon > 0$ small enough such that $\theta_{n} \varepsilon < 1$, we have 
  \begin{align}
  \label{eq:51}
    \cE^{\alpha}(u_{n}, u_{n}) \le \frac{-C(\theta_{n}) + \theta_{n} M(\varepsilon)}{1 - \theta_{n}\varepsilon}.
  \end{align}
  Therefore, we see
  \begin{align}
  \label{eq:52}
    \limsup_{n \to \infty} \cE^{\alpha}(u_{n}, u_{n})
    \le \frac{\theta_{c} M(\varepsilon)}{1 - \theta_{c} \varepsilon} < \infty.
  \end{align}
  With this bound, we evaluate the critical energy form for the sequence $u_n$:
  \begin{align}
  \label{eq:53}
    \left| \cE^{-\theta_{c}\mu}(u_{n}, u_{n}) + C(\theta_{n}) \right|
    &= \left| \cE^{-\theta_{c} \mu}(u_{n}, u_{n}) - \cE^{-\theta_{n} \mu}(u_{n},u_{n}) \right| \nonumber \\
    &= (\theta_{n} - \theta_{c}) \int_{\bR^{d}} u_{n}^{2} d\mu \nonumber \\
    &\le (\theta_{n} - \theta_{c}) \left\| G \mu \right\|_{\infty} \cE^{\alpha}(u_{n},u_{n}).
  \end{align}
  Since $\theta_n \downarrow \theta_c$ and the sequence $\cE^\alpha(u_n, u_n)$ is bounded,
  the right-hand side tends to $0$. Noting that $C(\theta_{n}) \to C(\theta_c) = 0$, we establish that
  \begin{align}
  \label{eq:54}
    \lim_{n \to \infty} \cE^{-\theta_{c}\mu}(u_{n},u_{n}) = 0.
  \end{align}

  By \eqref{eq:52}, the sequence $\{u_{n}\}$ is
  $\mathcal{E}^{\alpha}$-bounded in the extended Dirichlet space $\cF_{e}^{\alpha}$. 
  By Lemma \ref{lem:compact_embedding},
  the space $\cF_{e}^{\alpha}$ is compactly embedded into $L^{2}_{\text{loc}}(dx)$. 
  Furthermore, under $\mu \in \cK_{\infty}$,
  the embedding from $\cF_{e}^{\alpha}$ into $L^{2}(\mu)$ is also compact by 
  \cite[Theorem 4.9]{Takeda2019}.  
  Thus, we can extract a subsequence (still denoted by $\{u_{n}\}$)
  that converges weakly in $\cF_{e}^{\alpha}$ and strongly
  in both $L^{2}(\mu)$ and $L^{2}_{\text{loc}}(dx)$ to some function
  $u \in \cF_{e}^{\alpha}$. 
  By the lower semicontinuity and the compactness of
  the embedding $\cF_{e}^{\alpha} \hookrightarrow L^{2}(\mu)$, 
  we have $\cE^{-\theta_{c}\mu}(u, u) \le \liminf_{n \to \infty}
  \cE^{-\theta_{c}\mu}(u_{n}, u_{n}) = 0$.
  Since the critical form is non-negative, this forces $\cE^{-\theta_{c}\mu}(u, u) = 0$.

  Next, we apply Proposition \ref{prop:Oshima-ineq} to identify the limit $u$. 
  Let $\psi \in C_{0}^{\infty}(\bR^{d})$ be the smooth cutoff function
  with compact support defined in the proposition. Substituting $u_n$ into \eqref{eq:21} yields
  \begin{align}
  \label{eq:oshima_un}
    \int_{\bR^{d}} \left| u_{n}(x) - h(x) L(u_{n}/h) \right| g(x) h(x) dx
    \le C \cE^{-\theta_{c}\mu}(u_{n}, u_{n})^{1/2}.
  \end{align}
  By \eqref{eq:54}, the right-hand side vanishes as $n \to \infty$. 
  Recall that $L(u_{n}/h) = \int_{\bR^{d}} u_{n}(x) \psi(x) h(x) dx$. 
  Since $\psi$ has compact support and $u_{n}$ converges to
  $u$ strongly in $L^{2}_{\text{loc}}(dx)$,
  we can find that 
  \begin{align*}
    \lim_{n \to \infty} L(u_{n}/h) = \int_{\bR^{d}} u(x) \psi(x) h(x) dx =: \gamma.
  \end{align*}
  Since $u_n \to u$ almost everywhere along a further subsequence,
  applying Fatou's lemma to the left-hand side of \eqref{eq:oshima_un} gives
  \begin{align*}
    \int_{\bR^{d}} \left| u(x) - \gamma h(x) \right| g(x) h(x) dx 
    &\le \liminf_{n \to \infty} \int_{\bR^{d}} \left| u_{n}(x) - h(x) L(u_{n}/h) \right|
      g(x) h(x) dx \\
    &\le \lim_{n \to \infty} C \mathcal{E}^{-\theta_{c} \mu}(u_{n},u_{n})^{1/2} = 0.
  \end{align*}
  Since $g(x)$ and $h(x)$ are strictly positive almost everywhere,
  this implies $u(x) = \gamma h(x)$ a.e. 
  Recall that the eigenfunctions $\{u_{n}\}$ are
  $L^{2}(dx)$-normalized ($\int_{\bR^{d}} u_{n}^{2} dx = 1$). 
  Fatou's lemma yields 
  \begin{align*}
    \gamma^{2} \int_{\bR^{d}} h^{2} dx \le \liminf_{n \to \infty} \int_{\bR^{d}} u_{n}^{2} dx = 1.
  \end{align*}
  If $\gamma \not= 0$, this implies $h \in L^{2}(dx)$, which contradicts the null
  criticality of $\cH^{\theta_{c} \mu}$. Thus, we must have $c = 0$, and
  consequently the limit function $u$ satisfies $u = 0$ $m$-a.e. 

  Since $u \in \mathcal{F}_e^\alpha$ is quasi-continuous,
  the equality $u = 0$ $m$-a.e. implies $u = 0$ q.e.
  by \cite[Theorem 4.2.2, Lemma 4.1.5]{FOT}. 
  Furthermore, since the smooth measure $\mu \in \mathcal{K}_{\infty}$
  charges no set of capacity zero,
  we obtain $u = 0$ $\mu$-a.e. 
  Recall that the subsequence $\{u_n\}$ was extracted so that it converges strongly to $u$ in $L^2(\mu)$
  the compactness of the embedding $\cF_{e}^{\alpha} \hookrightarrow L^{2}(\mu)$. 
Since the limit function $u$ is $0$ in $L^2(\mu)$,
the strong convergence directly yields
\begin{align}
\label{eq:u_n_L2_zero}
  \lim_{n\to\infty} \int_{\mathbb{R}^d} u_{n}^{2} d\mu = \int_{\mathbb{R}^{d}}^{ }u^{2} d\mu = 0.
\end{align}

Finally, we return to \eqref{eq:53}.
Neglecting $\mathcal{E}^{-\theta_{c} \mu}(u_{n},u_{n}) \ge 0$, \eqref{eq:53} implies
\begin{align}
\label{eq:17}
  C(\theta_{n}) &\le \mathcal{E}^{-\theta_{c} \mu}(u_{n},u_{n}) + C(\theta_{n})
                  \le (\theta_{n} - \theta_{c}) \int_{\mathbb{R}^{d}}^{ } u_{n}^{2} d\mu.
\end{align}
Dividing both sides by the strictly positive quantity $(\theta_{n} - \theta_{c})$, we observe that
  \begin{align}
  \label{eq:derivative_bound}
    0 \le \frac{C(\theta_{n}) - C(\theta_{c})}{\theta_{n} - \theta_{c}}
    = \frac{C(\theta_{n})}{\theta_{n} - \theta_{c}} \le \int_{\bR^{d}} u_{n}^{2} d\mu.
  \end{align}
  Taking the limit as $n \to \infty$ and using \eqref{eq:u_n_L2_zero}, we obtain
  \begin{align*}
    \lim_{n \to \infty} \frac{C(\theta_{n}) - C(\theta_{c})}{\theta_{n} - \theta_{c}} = 0.
  \end{align*}
  Since the sequence $\{\theta_{n}\}$ decreasing to $\theta_{c}$ was arbitrary,
  we conclude that the principal eigenvalue curve $C(\theta)$ is right-differentiable
  at $\theta_{c}$ with its derivative being exactly zero. 
  Therefore the proof is completed. 
\end{proof}

\section{Large deviation principle}
In this section, we establish the large deviation principle for
the positive continuous additive functional corresponding to $\mu$.
First we identify the logarithmic moment generating function
$\Lambda(\theta)$ of the additive functional $A_{t}^{\mu}$ for the GSP.
To do this, we use the $L^{p}$-independence of the spectral radius
for the Feynman-Kac semigroup $\{P_{t}^{\theta_{c} \mu}\}$.

The generator of the GSP is defined by 
\begin{align*}
  \mathcal{H}^{\alpha} = -\log(1 + (-\Delta)^{\alpha/2}), \quad \alpha \in (0, 2],
\end{align*}
and the associated Dirichlet form is denoted by $(\mathcal{E}^{\alpha}, \mathcal{F}^{\alpha})$. 
For a measure $\mu \in \mathcal{K}_{\infty}$ and $\theta \in \mathbb{R}$, we consider the Feynman-Kac semigroup $\{P_{t}^{\theta\mu}\}_{t \ge 0}$. 
We first define the logarithmic moment generating function $\Lambda(\theta)$ by
\begin{align*}
  \Lambda(\theta) := \lim_{t \to \infty} \frac{1}{t} \log \mathbb{E}_{x} \left[ \exp(\theta A_{t}^{\mu}) \right],
\end{align*}
provided that the limit exists. 

\begin{definition}
  For $p \in [1, \infty]$, we define the $L^{p}$-spectral radius $\lambda_{p}(\theta\mu)$ as
  \begin{align*}
    \lambda_{p}(\theta\mu) := - \lim_{t \to \infty} \frac{1}{t} \log \|P_{t}^{\theta\mu}\|_{p,p}.
  \end{align*}
\end{definition}

\begin{lemma}
  The $L^{2}$-spectral radius $\lambda_{2}(\theta\mu)$ is characterized by the following variational formula:
  \begin{align*}
    \lambda_{2}(\theta\mu) = \inf \left\{ \mathcal{E}^{\alpha}(u,u) - \theta \int_{\mathbb{R}^{d}} u^{2} d\mu : u \in \mathcal{F}^{\alpha}, \int_{\mathbb{R}^{d}} u^{2} dx = 1 \right\}.
  \end{align*}
  Hence it follows that $C(\theta) := -\lambda_{2}(\theta\mu)$.
\end{lemma}
\begin{proof}
  Since the generator of the Feynman-Kac semigroup is $\mathcal{H}^{\alpha} + \theta\mu$,
  its $L^{2}$-spectral bound is the negative of the bottom of the spectrum of
  the associated quadratic form $\mathcal{E}^{\alpha}(u,u) - \theta \int u^{2} d\mu$.
  By the definition of $\lambda_{2}(\theta\mu)$, it coincides with this infimum.
\end{proof}

\begin{lemma}
  \label{lem:p-ind-1}
  For any $\theta \in \mathbb{R}$, the $L^{p}$-spectral radii $\lambda_{p}(\theta\mu)$ are independent of $p \in [1, \infty]$. That is, we have
  \begin{align*}
    -\lambda_{\infty}(\theta\mu) = -\lambda_{2}(\theta\mu) = C(\theta).
  \end{align*}
\end{lemma}

\begin{proof}
  By Lemma \ref{lem:spectral_properties}, we have already known that
  $C(\theta) \ge 0$ for all $\theta \in \bR$, which is equivalent to
  $\lambda_{2}(\theta \mu) \le 0$. Under this condition and $\mu \in \cK_{\infty}$,
  the $p$-independence of $\lambda_{p}(\theta \mu)$ holds by \cite[Theorem 2.4 (i)]{chen2012lp}.
\end{proof}

\begin{lemma}
  \label{lem:p-ind-2}
  For any $x \in \mathbb{R}^{d}$ and $\theta \in \mathbb{R}$, the following inequality holds:
  \begin{align*}
    \limsup_{t \to \infty} \frac{1}{t} \log \mathbb{E}_{x} \left[ \exp(\theta A_{t}^{\mu}) \right] \le -\lambda_{\infty}(\theta\mu).
  \end{align*}
\end{lemma}
\begin{proof}
  Using $\mathbb{E}_{x} [ \exp(\theta A_{t}^{\mu}) ] = P_{t}^{\theta\mu} 1(x)$,
  the expectation is bounded by the operator norm $\|P_{t}^{\theta\mu}\|_{\infty,\infty}$.
  The assertion follows directly from the definition of $\lambda_{\infty}(\theta\mu)$.
\end{proof}

\begin{lemma}\label{lem:C_BR_limit}
  Let $B(R)$ be the open ball of radius $R$ centered at the origin. 
  Let $C^{B(R)}(\theta)$ be the spectral function restricted to the open ball $B(R)$,
  that is,
  \begin{align}
  \label{eq:spec-2}
    C^{B(R)}(\theta) := -\inf \left\{ \mathcal{E}^{\alpha}(u,u) - \theta \int_{B(R)}^{ } u^{2} d\mu :
    u \in \mathcal{F}^{B(R)},\ \int_{B(R)}^{ } u^{2} dx = 1\right\},
  \end{align}
  where
  \begin{align}
  \label{eq:85}
    \mathcal{F}^{B(R)} := \left\{ u \in \mathcal{F}^{\alpha}: u = 0\ \text{$m$-a.e. on } B(R)^{c} \right\}.
  \end{align}
  Then, we have
\begin{align}
  \lim_{R\to\infty} C^{B(R)}(\theta) = C(\theta). 
\end{align}
\end{lemma}
\begin{proof}
  Since any function $u \in \mathcal{F}^{B(R)}$ can be extended by zero outside $B(R)$
  to become an element of $\mathcal{F}^{\alpha}$,
  we have the natural inclusion $\mathcal{F}^{B(R)} \subset \mathcal{F}^{\alpha}$. 
  This inclusion immediately yields the monotonicity
  $C^{B(R)}(\theta) \le C^{B(R')}(\theta) \le C(\theta)$ for $R < R'$, which provides the upper bound:
  \begin{align}
    \label{eq:C_upper}
    \lim_{R\to\infty} C^{B(R)}(\theta) \le C(\theta).
  \end{align}

  To prove the reverse inequality, fix an arbitrary $\varepsilon > 0$.
  By the definition of $C(\theta)$,
  there exists a function $u \in \mathcal{F}^{\alpha}$ with $\|u\|_{L^2(\mathbb{R}^d)} = 1$ such that
  \begin{align}
    \label{eq:C_eps}
  \theta \int_{\mathbb{R}^d} u^2 d\mu - \mathcal{E}^{\alpha}(u, u) > C(\theta) - \varepsilon.
  \end{align}
  Since $C_0^\infty(\mathbb{R}^d)$ is a core for the Dirichlet space $\mathcal{F}^{\alpha}$,
  it is dense in $\mathcal{F}^{\alpha}$ with respect to the $\mathcal{E}_{1}^{\alpha}$-norm. 
  Furthermore, because $\mu \in \mathcal{K}_{\infty}$,
  the embedding from $\mathcal{F}^{\alpha}$ to $L^2(d\mu)$ is continuous,
  which guarantees that the functional
  $f \mapsto \theta \int_{\mathbb{R}^d} f^2 d\mu - \mathcal{E}^{\alpha}(f, f)$ is continuous with respect to
  the $\mathcal{E}_{1}^{\alpha}$-norm. 
  Thus, we can choose a test function $v \in C_0^\infty(\mathbb{R}^d)$ close enough to $u$
  such that the $L^{2}$-normalized function $w = v / \|v\|_{L^2(\mathbb{R}^d)} \in C_0^\infty(\mathbb{R}^d)$
  satisfies $\|w\|_{L^2(\mathbb{R}^d)} = 1$ and 
  \begin{align}\label{eq:C_w}
    \theta \int_{\mathbb{R}^d} w^2 d\mu - \mathcal{E}^{\alpha}(w, w) > C(\theta) - 2\varepsilon.
  \end{align}
  Since $w$ has a compact support, there exists $R_0 > 0$ such that $\supp(w) \subset B(R_0)$. 
  For any $R \ge R_0$, $w$ vanishes outside $B(R)$, meaning $w \in \mathcal{F}^{B(R)}$
  and $\int_{B(R)} w^2 dx = \int_{\mathbb{R}^d} w^2 dx = 1$. 
  Substituting $w$ as a test function into \eqref{eq:spec-2}, we obtain for all $R \ge R_0$,
\begin{align}
  C^{B(R)}(\theta) &\ge \theta \int_{B(R)} w^2 d\mu - \mathcal{E}^{\alpha}(w, w) \notag \\
  &= \theta \int_{\mathbb{R}^d} w^2 d\mu - \mathcal{E}^{\alpha}(w, w) \notag \\
  &> C(\theta) - 2\varepsilon.
\end{align}
Taking the limit as $R \to \infty$ yields
\begin{align}
  \lim_{R\to\infty} C^{B(R)}(\theta) \ge C(\theta) - 2\varepsilon.
\end{align}
Since $\varepsilon > 0$ is arbitrary, we conclude $\lim_{R\to\infty} C^{B(R)}(\theta) \ge C(\theta)$. 
Combining this with \eqref{eq:C_upper}, the proof is complete.
\end{proof}

\begin{theorem}
  \label{thm:LMGFSP}
  For any $x \in \mathbb{R}^{d}$ and $\theta \in \mathbb{R}$,
  the logarithmic moment generating function $\Lambda(\theta)$ exists
  and coincides with the spectral function $C(\theta)$:
  \begin{align*}
    \Lambda(\theta) = C(\theta).
  \end{align*}
\end{theorem}
\begin{proof}
  By Lemma \ref{lem:p-ind-1} and Lemma \ref{lem:p-ind-2},
  we already have the upper bound:
\begin{align}\label{eq:upper_bound}
  \Lambda(\theta)
  \le \limsup_{t\to\infty} \frac{1}{t} \log \mathbb{E}_x[\exp(\theta A_t^\mu)]
  &\le -\lambda_\infty(\theta \mu) = C(\theta).
\end{align}

To establish the lower bound, fix an arbitrary starting point $x \in \mathbb{R}^d$.
We use the killed process on a bounded domain 
to ensure the compactness of the semigroup. 
For any $R > |x|$, let $B(R)$ be the open ball of radius $R$ centered at the origin, 
and let $\tau_{B(R)}$ be the first exit time from $B(R)$. 
Let $P_t^{\theta\mu, B(R)}$ be the Feynman-Kac semigroup killed upon leaving $B(R)$.
Then it holds that
\begin{align}
\label{eq:86}
  C^{B(R)}(\theta) = \lim_{t \to \infty} \frac{1}{t} \log \|P_{t}^{\theta \mu, B(R)}\|_{2,2}.
\end{align}

By Lemma \ref{lem:compact_embedding}, the operator $P_t^{\theta\mu, B(R)}$ is compact on $L^2(B(R), dx)$, yielding a ground state $h_R \in L^2(B(R), dx)$. 
Although the GSP lacks ultracontractivity, we can deduce that $h_R$ is bounded by using its fine continuity and the boundedness of the Kato class potential (similarly to \cite[Lemma 4.9]{TakedaTsuchida2007}). 
Once the boundedness of $h_R$ is established, the strong Feller property of the GSP ensures that $h_R$ has a continuous version. 
Furthermore, the irreducibility guarantees that $h_R$ is strictly positive on $B(R)$.
Let $M_R = \|h_R\|_{L^\infty(B(R))} < \infty$. 
Since $1 \ge h_R(y)/M_R$ for all $y \in B(R)$, 
we have the trivial lower bound for any fixed $x \in B(R)$:
\begin{align}
  \mathbb{E}_x[\exp(\theta A_t^\mu)]
  &\ge \mathbb{E}_x[\exp(\theta A_t^\mu),\ t < \tau_{B(R)}] 
  = P_t^{\theta\mu, B(R)} 1(x) \\
  &\ge P_t^{\theta\mu, B(R)} \left(\frac{h_R}{M_R}\right)(x)
    = \frac{1}{M_R} e^{t C^{B(R)}(\theta)} h_R(x). 
\end{align}
Since $h_R(x) > 0$, taking the logarithm, dividing by $t$, and letting $t \to \infty$ yields
\begin{align}
  \liminf_{t\to\infty} \frac{1}{t} \log \mathbb{E}_x[\exp(\theta A_t^\mu)] &\ge C^{B(R)}(\theta).
\end{align}
By taking $R \to \infty$ and using Lemma \ref{lem:C_BR_limit}, we obtain the global lower bound:
\begin{align}\label{eq:lower_bound}
 \Lambda(\theta) \ge \liminf_{t\to\infty} \frac{1}{t} \log \mathbb{E}_x[\exp(\theta A_t^\mu)] &\ge C(\theta).
\end{align}

Combining the upper bound \eqref{eq:upper_bound} and the lower bound \eqref{eq:lower_bound}, 
we conclude that the limit exists and 
\begin{align}
  \Lambda(\theta) = \lim_{t\to\infty} \frac{1}{t} \log \mathbb{E}_x[\exp(\theta A_t^\mu)]
  &= C(\theta).
\end{align}
This completes the proof.
\end{proof}

Finally, as an application of differentiability of spectral function,
we get the following large deviation principle.

\begin{theorem}
  \label{thm:LDP}
  Assume that  $\alpha < d \le 2\alpha$ and $\mu \in \mathcal{K}_{\infty}$. 
  The large deviation principle for $A_{t}^{\mu}$ holds:
  for any open set $G$ and closed set $F$ in $\bR$, 
  \begin{align}
  \label{eq:65}
    & \liminf_{t \to \infty} \frac{1}{t} \log \bP_{x} \left( \frac{A_{t}^{\mu}}{t} \in G \right)
      \ge -\inf_{\lambda \in G} I(\lambda), \\
    & \limsup_{t \to \infty} \frac{1}{t} \log \bP_{x} \left( \frac{A_{t}^{\mu}}{t} \in F \right)
      \le -\inf_{\lambda \in F} I(\lambda), 
  \end{align}
  where $I(\lambda)$ is the Fenchel-Legendre transform of $C(\theta)$, that is,
  \begin{align}
  \label{eq:66}
  I(\lambda) = \sup_{\theta \in \bR} \{\lambda \theta - C(\theta)\}.
  \end{align}
\end{theorem}
\begin{proof}
  By Theorem \ref{thm:LMGFSP}, we know that the logarithmic moment generating function
  $\Lambda(\theta)$ coincides with the spectral function $C(\theta)$.
  Hence, by Theorem \ref{thm:main}, $\Lambda(\theta)$ is finite and
  differentiable on $\bR$.
  The G\"artner-Ellis theorem (see \cite[Theorem 2.3.6]{DemboZeitouni2009})
  implies that the large deviation principle for $A_{t}^{\mu}$ holds.
\end{proof}

\end{document}